\documentclass[12pt]{amsart}

\usepackage{graphicx}
\usepackage[english]{babel}
\usepackage{graphicx}
\usepackage{framed}
\usepackage[normalem]{ulem}
\usepackage{amsmath}
\usepackage{amsthm}
\usepackage{amssymb}
\usepackage[bookmarks=true, bookmarksopen=true, bookmarksdepth=3,bookmarksopenlevel=2, colorlinks=true, linkcolor=blue, citecolor=blue, filecolor=blue, menucolor=blue, urlcolor=blue]{hyperref}
\usepackage[capitalize]{cleveref}
\usepackage{comment}
\usepackage{tikz}
\usetikzlibrary{matrix,calc}
\usepackage{mathtools}
\usepackage{tikz-cd}
\usepackage{amsfonts}
\usepackage{enumerate}
\usepackage[utf8]{inputenc}
\usepackage[margin=1in,marginparwidth=0.8in, marginparsep=0.1in]{geometry}
\usepackage{stmaryrd}
\usepackage{bbm}
       
\usepackage{tikz}
\usetikzlibrary{3d,arrows,calc,positioning,decorations.pathreplacing,matrix} 
\newcommand{\arrtip}{latex'}

    \usepackage[colorinlistoftodos]{todonotes}
    \usepackage[all]{xy}
    \usepackage{mathrsfs}

    \newcommand{\C}{\mathbb{C}} 
    \newcommand{\N}{\mathbb{N}} 
    \newcommand{\G}{\mathbb{G}} 
    
    \newcommand{\SL}{\operatorname{SL}}
     \newcommand{\GL}{\operatorname{GL}}

    \newcommand{\LT}{\mathfrak{t}}

    \newcommand{\id}{\operatorname{id}} 
    \newcommand{\Hom}{\operatorname{Hom}}
    
    \renewcommand{\Hom}{\operatorname{Hom}} 
    
    \newcommand{\IndCoh}{\operatorname{IndCoh}}
    
    \newcommand{\D}{\mathcal{D}} 
    \renewcommand{\\}{\backslash}
    \renewcommand{\subset}{\subseteq}
    
    \numberwithin{equation}{section}
    \newtheorem{Theorem}[equation]{Theorem}
    \newtheorem{Proposition}[equation]{Proposition}

    \newtheorem{Corollary}[equation]{Corollary}

    \theoremstyle{definition}
    \newtheorem{Remark}[equation]{Remark}

    \numberwithin{figure}{section}

    \title[Differential operators on $SL_n/U$ and quantized Coulomb branches]{Differential operators on the base affine space of $SL_n$ and quantized Coulomb branches}
    
    \author[Tom Gannon]{Tom Gannon}
    \address[Tom Gannon]{University of California, Los Angeles \\ Los Angeles CA, USA}
    \email{gannonth@math.ucla.edu}
    
    \author[Harold Williams]{Harold Williams}
    \address[Harold Williams]{University of Southern California \\ Los Angeles CA, USA}
    \email{hwilliams@usc.edu}

    \newcommand{\AvN}{\text{Av}_*^N}
    \newcommand{\Avpsi}{\text{Av}_{!}^{\psi}}

    \newcommand{\LTd}{\LT^{\ast}}

\DeclareFontFamily{U}{mathx}{\hyphenchar\font45}
\DeclareFontShape{U}{mathx}{m}{n}{
	<5> <6> <7> <8> <9> <10>
	<10.95> <12> <14.4> <17.28> <20.74> <24.88>
	mathx10
}{}
\DeclareSymbolFont{mathx}{U}{mathx}{m}{n}
\DeclareFontSubstitution{U}{mathx}{m}{n}
\DeclareMathAccent{\widecheck}{0}{mathx}{"71}

\newcommand{\Sp}{\mathrm{Sp}}
\DeclareMathSymbol{\shortminus}{\mathbin}{AMSa}{"39}
\DeclareMathSymbol{-}{\mathbin}{AMSa}{"39}
    
\begin{document}

\newcommand{\AvNTw}{\text{Av}_*^{N, (T, w)}}
\newcommand{\AvGw}{\text{Av}_{\ast}^{G,w}}
\newcommand{\pifin}{\pi_{\text{fin}}}
\newcommand{\pifinL}{\pi_{\text{fin,L}}}

    \newcommand{\ELeftAdjoint}{\text{ev}_{\omega_{\LTd}}}
\newcommand{\ClGlobalDiffOp}{\text{H}^0\Gamma(\mathcal{D}_{G/N})}
\newcommand{\GlobalDiffOp}{\Gamma(\mathcal{D}_{G/N})}
\newcommand{\indsch}{\mathcal{X}}

    \newcommand{\DGCatContk}{\text{DGCat}^k_{\text{cont}}}
\newcommand{\DGCatContL}{\text{DGCat}^L_{\text{cont}}}

\newcommand{\DNTw}{\mathcal{D}(N\backslash G/N)^{T_r,w}}
\newcommand{\DNTwWhit}{\mathcal{D}(N^-_{\psi}\backslash G/N)^{T_r,w}}
\newcommand{\DNWhit}{\mathcal{D}(N^-_{\psi}\backslash G/N)}
\newcommand{\DNTwldeg}{\mathcal{D}(N \backslash G/N)^{T_r,w}_{\text{left-deg}}}
\newcommand{\DNTwnondeg}{\mathcal{D}(N \backslash G/N)^{T_r,w}_{\text{nondeg}}}
\newcommand{\DN}{\mathcal{D}(N\backslash G/N)}
\newcommand{\DNldeg}{\mathcal{D}(N \backslash G/N)_{\text{left-deg}}}
\newcommand{\DNnondeg}{\mathcal{D}(N \backslash G/N)_{\text{nondeg}}}
\newcommand{\DNlambda}{\mathcal{D}^{\lambda}(N \backslash G/B)}
\newcommand{\Dpsilambda}{\mathcal{D}^{\lambda}(N^- _{\psi}\backslash G/B)}
\newcommand{\DbiTw}{\mathcal{D}(N \backslash G/N)^{T \times T, w}}
\newcommand{\DbiTwnondeg}{\mathcal{D}(N \backslash G/N)^{T \times T, w}_{\text{nondeg}}}
\newcommand{\DbiTwnondegheart}{\mathcal{D}(N \backslash G/N)^{(T \times T, w), \heartsuit}_{\text{nondeg}}}
\newcommand{\DbiTwdeg}{\mathcal{D}(N \backslash G/N)^{T \times T, w}_{\text{deg}}}
\newcommand{\DNBlambda}{\mathcal{D}(N \backslash G/_{\lambda}B)}
\newcommand{\DNTwBlambda}{\mathcal{D}(N \backslash G/_{\lambda}B)}
\newcommand{\DWhitBlambda}{\mathcal{D}(N^-_{\psi}\backslash G/_{\lambda}B)}
\newcommand{\HN}{\D(N \backslash G/N)}
\newcommand{\HNTw}{\D(N \backslash G/N)^{T \times T, w}}
\newcommand{\HNTwabbreviated}{\mathcal{H}^{N, (T,w)}}
\renewcommand{\indsch}{\mathcal{X}}
\newcommand{\wdot}{\dot{w}}
\newcommand{\Gaminusalpha}{\mathbb{G}_a^{-\alpha}}
\newcommand{\Aone}{\mathbf{A}}
\newcommand{\Atwo}{\mathcal{L}\text{-mod}(\Aone)}
\newcommand{\algobj}{\mathcal{A}}
\newcommand{\newalgobj}{\mathcal{A}'}
\newcommand{\tilder}{\tilde{r}}
\newcommand{\Ccirc}{\mathring{\C}}
\newcommand{\rootlattice}{\mathbb{Z}\Phi}
\newcommand{\characterlatticeforT}{X^{\bullet}(T)}
\newcommand{\FourierMukai}{\text{FMuk}}
\newcommand{\oneshiftedCartierdual}{c_1}

\newcommand{\tildeV}{\tilde{\mathbb{V}}}
\newcommand{\Vdual}{V^{\vee}}
\newcommand{\generalstacktoGITquotientmap}{\phi}
\newcommand{\SpecofL}{\text{Spec}(L)}
\newcommand{\terminalmapfromC}{\alpha}
\newcommand{\terminalmapfromCmodassociatedstabilizer}{\dot{\alpha}}
\newcommand{\Hpsiliteral}{\mathcal{H}_{\psi}}
\newcommand{\AvNshifted}{\AvN[\text{dim}(N)]}
\newcommand{\hyperplanefixedbys}{V^{\ast}_{s = \text{id}}}
\newcommand{\fieldpossiblydifferentfromgroundfield}{K}
\newcommand{\Avpsishifted}{\Avpsi[-\text{dim}(N)]}
\newcommand{\FI}{F_{I}}
\newcommand{\FD}{F_{\mathcal{D}}}
\newcommand{\LI}{s_*^{\IndCoh}}
\newcommand{\LD}{\Avpsi}
\newcommand{\LIenh}{\pi_*^{\IndCoh, \text{enh}}}
\newcommand{\LDenh}{\text{Av}_!^{\psi, \text{enh}}}
\newcommand{\Ind}{\text{Ind}}
\newcommand{\Res}{\text{Res}}
\newcommand{\IndWithoutSignRep}{\text{Ind}(-)^W}
\newcommand{\ResWithoutSignRep}{\text{WRes}}
\newcommand{\IndWithSignRep}{\text{Ind}(- \otimes k_{\text{sign}})^W}
\newcommand{\ResWithSignRep}{\text{WRes}_s}
\newcommand{\parabolicrestrictionLIFTED}{\text{WRes}}
\newcommand{\V}{\mathcal{V}}
\newcommand{\parabolicrestriction}{\text{Res}}

\newcommand{\conjugacyclassofstandardLevis}{\underline{\Theta}}
\newcommand{\moduliOfEvenMonicPolynomialsOfTHISDEGREE}[1]{\mathcal{M}^{#1}}
\newcommand{\moduliOfALLEvenPolynomialsOfTHISDEGREE}[1]{\mathcal{P}^{#1}}
\newcommand{\moduliOfEvenMonicPolynomialsOfdegreeTWOwithMINUSSIGN}{\mathcal{M}_{-}^{2}}

\newcommand{\basedQuasimapsfromProjectiveLinetoAffineClosureofSLTWOwithIsoClassofTHISDEGREE}[1]{\text{Maps}_*^{#1}(\mathbb{P}^1, \overline{\SL_2/N}/T)}
\newcommand{\basedQuasimapsfromProjectiveLinetoAffineClosureofSLTWOwithISOtoTHISDEGREE}[1]{\text{Maps}_*^{#1, \simeq}(\mathbb{P}^1, \overline{\SL_2/N}/T)}
\newcommand{\oblv}{\text{oblv}}
\newcommand{\explicitIsoofBasedQMapsforSLTWOwithPolynomialsforMapsofTHISDEGREE}[1]{\Phi_{#1}}
\newcommand{\basedQuasimapsfromProjectiveLinetoAffineClosureofBORELwithIsoClassofTHISDEGREE}[1]{\text{Maps}_*^{#1}(\mathbb{P}^1, \overline{B/N}/T)}

\newcommand{\basedmapsFromProjectiveLinetoTHISSPACE}[1]{\text{Maps}_*(\mathbb{P}^1, #1)}

\newcommand{\resolutionOfSingularitiesSpaceForBasicAffineSpace}{G\mathop{\times}\limits^{B} E}

\newcommand{\affineClosureofBasicAffineSpace}{\overline{G/N}}
\newcommand{\affineClosureOfCotangentBundleofBasicAffineSpace}{\overline{T^*(G/N)}}
\newcommand{\smoothLocusOfAffineclosureofCotangentBundleOfBasicAffineSpace}{\overline{T^*(G/N)}_{\text{reg}}}
\newcommand{\SpecOfSymofSectionsOfTangentBundleOfBasicAffineSpace}{\text{Spec}(\text{Sym}_A^{\bullet}(\mathcal{T}_{\overline{G/N}}))}
\newcommand{\ringOfFunctionsForBasicAffineSpace}{A}
\newcommand{\ringOfFunctionsForCOTANGENTBUNDLEOfBasicAffineSpace}{R}
\newcommand{\tangentSheafForBasicAffineSpace}{\mathcal{T}_{G/N}}
\newcommand{\symOfDirectSumOfRepsofFundamentalWeights}{\text{Sym}(\oplus_i E(\omega_i))}
\newcommand{\Sym}{\text{Sym}}
\newcommand{\projectionFromAffineClosureofCotangentBundleToAffineClosureofSpace}{\overline{\pi}}
\newcommand{\momentMapFromAFFINECLOSUREofCotangentSpaceWithGROUPG}{\overline{\mu}_G}
\newcommand{\momentMapFromAFFINECLOSUREofCotangentSpacewithTOTALGROUP}{\overline{\mu}}
\newcommand{\groundfield}{k}
\newcommand{\LGd}{\mathfrak{g}^*}
\newcommand{\momentMapFromAFFINECLOSUREofCotangentSpaceWithGROUPT}{\overline{\mu}_T}

\newcommand{\LeviSubgroupofGLn}{M}
\newcommand{\borelOfLevi}{B_L}
\newcommand{\unipotentRadicalofLevi}{N_L}
\newcommand{\lieAlgebraofUnipotentRadicalofLevi}{\mathfrak{n}_L}
\newcommand{\universalPartialHyperkahlerImplosionwithParabolic}{G \mathop{\times}\limits^{\unipotentRadicalOfPARABOLICSUBGROUP} (\mathfrak{g}/\lieAlgebraOfUnipotentRadicalOfPARABOLICSUBGROUP)^*}
\newcommand{\rhocheck}{\rho^{\vee}}
\newcommand{\unipotentRadicalOfPARABOLICSUBGROUP}{U_P}
\newcommand{\lieAlgebraOfUnipotentRadicalOfPARABOLICSUBGROUP}{\mathfrak{u}_P}
\newcommand{\affinizationOfGrothendieckSpringerResolution}{\LGd \times_{\LTd\sslash W} \LTd}
\newcommand{\smoothLocusOfAffineclosureBasicAffineSpace}{\overline{G/N}^{\text{sm}}}

\newcommand{\ringOfFunctionsForFIBERPRODUCTofCotangentBundleofGModNOverLieTWithZero}{\overline{\ringOfFunctionsForCOTANGENTBUNDLEOfBasicAffineSpace}}
\newcommand{\affineClosureFIBERPRODUCTOfCotangentBundleofBasicAffineSpaceatZeroOverLT}{\affineClosureOfCotangentBundleofBasicAffineSpace \times_{\LTd} \{0\}}
\newcommand{\WhittakerHamiltonianReductionofG}{T^{\ast}(G/_{\psi}N^-)}
\newcommand{\WhittakerHamiltonianReductionofGBASECHANGEDtoLTd}{\WhittakerHamiltonianReductionofG \times_{\LTd\sslash W}\LTd}
\newcommand{\RingOfFunctionsForWhittakerHamiltonianReductionOfGBASECHANGEDtoLTd}{S}
\newcommand{\affineGrassmannianForLANGLANDSDUALgroup}{\text{Gr}_{G^{\vee}}}
\newcommand{\regularRepresentationPerverseSheafOnAffineGrassmannianForLANGLANDSDUALgroup}{\mathcal{A}_G}
\newcommand{\nilHeckeRingForW}{\mathscr{N}}

\newcommand{\GLnhat}{\widehat{\GL_n}}
\newcommand{\quotientOfGLnhatByDiagonalCentralSubgroup}{\widetilde{\GL_n}}
\newcommand{\DiagonalMatricesinGLn}{D}
\newcommand{\DiagonalMatricesHat}{\widehat{\DiagonalMatricesinGLn}}
\newcommand{\quotientOfDiagonalMatricesHatByDiagonalCentralSubgroup}{\tilde{\DiagonalMatricesinGLn}}
\newcommand{\affineClosureOfCotangentBundleofBasicAffineSpaceFORSLN}{\overline{T^*(\SL_n/U)}}

\newcommand{\pt}{\mathrm{pt}}

\newcommand{\Gr}{\mathrm{Gr}}
\newcommand{\hGr}{\widehat{\mathrm{Gr}}}
\newcommand{\conv}{{\mathbin{\scalebox{1.1}{$\mspace{1.5mu}*\mspace{1.5mu}$}}}}

\newcommand{\oldtighttimes}{\mathbin{\!\times\!}}
\newcommand{\tighttimes}{\mathrel{\mkern-5mu\times\mkern-5mu}}
\newcommand{\tbox}{\mathbin{\widetilde{\boxtimes}}}
\newcommand{\ol}{\overline}
\def\Spec{{\rm{Spec}}\,}

\newcommand{\hG}{\widehat{G}}
\newcommand{\hL}{\widehat{L}}
\newcommand{\hT}{\widehat{T}}
\newcommand{\wG}{\widetilde{G}}
\newcommand{\wL}{\widetilde{L}}
\newcommand{\wT}{\widetilde{T}}
\newcommand{\wW}{\widetilde{W}}
\newcommand{\wN}{\widetilde{N}}
\newcommand{\wi}{\widetilde{i}}
\newcommand{\wj}{\widetilde{j}}
\newcommand{\wft}{\widetilde{\mathfrak{t}}}
\newcommand{\wone}{\widetilde{1}}

\newcommand{\kk}{k}
\newcommand{\Shv}{\mathrm{Shv}}
\newcommand{\Sch}{\mathrm{Sch}}
\newcommand{\Schft}{\mathrm{Sch}_{ft}}
\newcommand{\PStk}{\mathrm{PStk}}
\newcommand{\Stk}{\mathrm{Stk}}
\newcommand{\op}{\mathrm{op}}
\newcommand{\indGStk}{\mathrm{indGStk}}
\newcommand{\GStk}{\mathrm{GStk}}
\newcommand{\catC}{{\mathscr{C}}}
\newcommand{\catD}{{\mathscr{D}}}
\newcommand{\catE}{{\mathscr{E}}}
\newcommand{\PrL}{\mathcal{P}\mathrm{r}^{\mathrm{L}}}
\newcommand{\PrR}{\mathcal{P}\mathrm{r}^{\mathrm{R}}}
\newcommand{\PrSt}{\mathcal{P}\mathrm{r}^{\mathrm{St}}}
\newcommand{\into}{\hookrightarrow}
\newcommand{\bul}{\bullet}
\newcommand{\Aff}{\mathrm{Aff}}
\newcommand{\Affft}{\mathrm{Aff_{ft}}}
\newcommand{\Corr}{\mathrm{Corr}}
\newcommand{\colim}{\mathrm{colim}}
\newcommand{\Catinfty}{\mathrm{Cat_\infty}}

\newcommand{\bR}{\mathbb{R}}
\newcommand{\bN}{\mathbb{N}}
\newcommand{\bZ}{\mathbb{Z}}
\newcommand{\bQ}{\mathbb{Q}}
\newcommand{\bC}{\mathbb{C}}

\newcommand{\cA}{\mathcal{A}}
\newcommand{\cB}{\mathcal{B}}
\newcommand{\cC}{\mathcal{C}}
\newcommand{\cD}{\mathcal{D}}
\newcommand{\cE}{\mathcal{E}}
\newcommand{\cF}{\mathcal{F}}
\newcommand{\cG}{\mathcal{G}}
\newcommand{\cH}{\mathcal{H}}
\newcommand{\cI}{\mathcal{I}}
\newcommand{\cJ}{\mathcal{J}}
\newcommand{\cK}{\mathcal{K}}
\newcommand{\cL}{\mathcal{L}}
\newcommand{\cM}{\mathcal{M}}
\newcommand{\cN}{\mathcal{N}}
\newcommand{\cO}{\mathcal{O}}
\newcommand{\cP}{\mathcal{P}}
\newcommand{\cQ}{\mathcal{Q}}
\newcommand{\cR}{\mathcal{R}}
\newcommand{\cS}{\mathcal{S}}
\newcommand{\cT}{\mathcal{T}}
\newcommand{\cU}{\mathcal{U}}
\newcommand{\cV}{\mathcal{V}}
\newcommand{\cW}{\mathcal{W}}
\newcommand{\cX}{\mathcal{X}}
\newcommand{\cY}{\mathcal{Y}}
\newcommand{\cZ}{\mathcal{Z}}

\newcommand{\fg}{\mathfrak{g}}
\newcommand{\fu}{\mathfrak{u}}
\newcommand{\fl}{\mathfrak{l}}
\newcommand{\ft}{\mathfrak{t}}

\newcommand{\al}{\alpha}
\newcommand{\be}{\beta}
\newcommand{\ga}{\gamma}
\newcommand{\Ga}{\Gamma}
\newcommand{\de}{\delta}
\newcommand{\De}{\Delta}
\newcommand{\ep}{\epsilon}
\newcommand{\ze}{\zeta}
\newcommand{\tht}{\theta}
\newcommand{\Th}{\Theta}
\newcommand{\io}{\iota}
\newcommand{\ka}{\kappa}
\newcommand{\la}{\lambda}
\newcommand{\La}{\Lambda}
\newcommand{\si}{\sigma}
\newcommand{\Si}{\Sigma}
\newcommand{\om}{\omega}
\newcommand{\Om}{\Omega}

\newcommand{\commutativeAlgebraObjectInEquivariantSatakeCategoryForHTwo}{\mathcal{A}}
\newcommand{\loopsOnH}{H_{\mathcal{K}}}
\newcommand{\arcsOnH}{H_{\mathcal{O}}}
\newcommand{\loopsOnHPRIME}{H'_{\mathcal{K}}}
\newcommand{\arcsOnHPRIME}{H'_{\mathcal{O}}}
\newcommand{\quotientMapFromGLnHatQuotientToPGLn}{q}
\newcommand{\quotientMapFromDiagonalMatricesHatQuotietnToTorusOfPGLn}{p}
\newcommand{\commutativeAlgebraSheafinQuotientofGLnhatSatakeCategoryforANQuiver}{\mathcal{A}_{\quotientOfGLnhatByDiagonalCentralSubgroup, V}}
\newcommand{\commutativeAlgebraSheafinQuotientofDIAGONALMATRICEShatforANQuiver}{\mathcal{A}_{\quotientOfDiagonalMatricesHatByDiagonalCentralSubgroup, V}}
\newcommand{\commutativeAlgebraSheafofGLnhatPushedForwardToBeInPGLnSatakeCategory}{\mathcal{A}}
\newcommand{\regularRepresentationSheafforinGrPGLn}{\mathcal{A}_R}
\newcommand{\maximalTorusForPGLn}{\overline{D}}
\newcommand{\wt}{\widetilde}
\newcommand{\sX}{\mathscr{X}}
\newcommand{\sY}{\mathscr{Y}}

\begin{abstract}
We show that the algebra $D_\hbar(SL_n/U)$ of differential operators on the base affine space of $SL_n$ is the quantized Coulomb branch of a certain 3d $\mathcal{N} = 4$ quiver gauge theory. In the semiclassical limit this proves a conjecture of Dancer-Hanany-Kirwan about the universal hyperk\"ahler implosion of $SL_n$. We also formulate and prove a generalization identifying the Hamiltonian reduction $T^* SL_n \sslash_{\psi} U$ as a Coulomb branch for an arbitrary unipotent character $\psi$. As an application of our results, we provide a new interpretation of the Gelfand-Graev symmetric group action on $D_\hbar(SL_n/U)$. 
\end{abstract}

\maketitle

\setcounter{tocdepth}{1}

\tableofcontents

\section{Introduction}\label{sec:intro}

\thispagestyle{empty}

Let $G$ be a semisimple complex algebraic group and $U$ a maximal unipotent subgroup. 
A prominent role in Lie theory is played by the base affine space $G/U$. In particular, the rich structure of its algebra $D(G/U)$ of differential operators was emphasized by Bernstein-Gelfand-Gelfand in \cite{BGG75}, and has been a topic of continuing interest to the present day \cite{BBP02,LS06,GR15,GK22}. The first goal of this paper is to show that when $G = SL_n$ this algebra (or more specifically, its asymptotic enhancement $D_\hbar(G/U)$) admits the following interpretation in terms of $3d$ $\cN=4$ gauge theory.

\begin{Theorem}\label{thm:mainthmintro}
There is an algebra isomorphism $D_\hbar(SL_n/U) \cong \C_\hbar[\cM_C(Q_n)]$, where the latter denotes the quantized Coulomb branch of the $3d$ $\cN=4$ quiver gauge theory associated to the quiver 
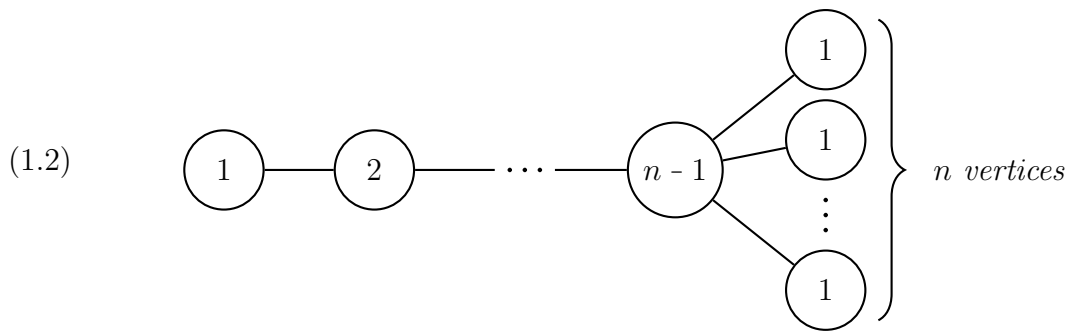
\begin{equation}\label{eq:quiverQn}
\begin{tikzpicture}[baseline=(current bounding box.center), thick, >=Stealth]
	\node[circle, draw, minimum size=30] (1) at (0,0) {$1$};
	\node[circle, draw, minimum size=30] (2) at (2,0) {$2$};
	\fill (3.8,0) circle (.03);
	\fill (4,0) circle (.03);
	\fill (4.2,0) circle (.03);
	\fill (8,-0.4) circle (.03);
	\fill (8,-0.6) circle (.03);
	\fill (8,-0.8) circle (.03);
	\node[circle, draw, minimum size=30] (n-1) at (6,0) {$n - 1$};
	\node [circle, draw, minimum size=30] (v1) at (8,1.6) {$1$};
	\node [circle, draw, minimum size=30] (v2) at (8,0.4) {$1$};
	\node [circle, draw, minimum size=30] (v4) at (8,-1.6) {$1$};
	
	\draw (2) -- (1);
	\draw (2) -- (3.6,0);
	\draw (n-1) -- (4.4,0);
	\draw (n-1) -- (v1);
	\draw (n-1) -- (v2);
	\draw (n-1) -- (v4);
	
	\draw[decorate, decoration={brace, amplitude=10pt}] (8.7, 2) -- (8.7, -2); 
	\node at (10.3,0) {$n$ vertices};
\end{tikzpicture}
\end{equation} and gauge group $\tilde{T}$ defined by (\ref{eq:Tgroupsquare}). 
	\end{Theorem}

Coulomb branches are geometric objects attached to certain supersymmetric quantum field theories. In cases such as the quiver gauge theory above, Braverman-Finkelberg-Nakajima described the Coulomb branch in terms of the Borel-Moore homology of a certain infinite-dimensional space \cite{BFN18}. Specifically, to $Q_n$ is associated a pair of an algebraic group~$G$ and a representation $N$, and to any such pair they associate a space $\cR_{G,N}$. This has an action by the jet group $G_\cO$, and $\cM_C(Q_n)$ is the affine variety whose coordinate ring is $H_{\bullet}^{G_\cO}(\cR_{G,N})$. This has a commutative multiplication induced by a form of convolution, which becomes quantized when we incorporate equivariance under loop rotation.  A diverse range of objects in geometry and Lie theory turn out to be the Coulomb branches of corresponding field theories, and Theorem \ref{thm:mainthmintro} adds a new entry to the catalogue of such examples. 

A feature of $D_\hbar(G/U)$ which figures prominently in the literature is an action of the Weyl group of $G$, originally discovered by Gelfand-Graev. This is not induced from an action on the variety $G/U$, but rather is generated by a collection of partial Fourier transforms corresponding to the simple roots of $G$. That these generators satisfy braid relations is not obvious, and understanding them better has been a topic of recent interest \cite{GK22}. It is thus an attractive feature of Theorem \ref{thm:mainthmintro} that, in the $SL_n$ case, this action is completely manifest from the Coulomb branch perspective. 

\begin{Theorem}\label{thm:Weylaction}
Under the isomorphism of Theorem \ref{thm:mainthmintro}, the Gelfand-Graev action of $S_n$ on $D_\hbar(SL_n/U)$ is identified with the action on $\C_\hbar[\cM_C(Q_n)]$ induced by permutations of the vertices on the right-hand side of (\ref{eq:quiverQn}). 
\end{Theorem}

A more recent motivation for the study of $D_\hbar(G/U)$ comes from differential geometry, specifically the Dancer-Kirwan-Swann theory of hyperk\"ahler implosions \cite{DKS13}. Recall that $D_\hbar(G/U)$ is a quantization of the ring of regular functions on $T^*(G/U)$. Since the variety $G/U$ is not affine but only quasi-affine, $T^*(G/U)$ is not equal to its affine closure $\ol{T^*(G/U)} := \Spec\, \C[T^*(G/U)]$. The latter contains $T^*(G/U)$ as a smooth open subvariety, but globally is typically complicated and singular  (indeed, it is nontrivial even to show that $\C[T^*(G/U)]$ is finitely generated \cite[Lem. 3.6.2]{GR15}).  Dancer-Kirwan-Swann argue that 
$\ol{T^*(G/U)}$ can be interpreted as the universal hyperk\"ahler implosion of $G$, analogous to the earlier theory of symplectic implosions \cite{GJS02}. 

Theorem \ref{thm:mainthmintro} is directly inspired by more recent work of Dancer-Hanany-Kirwan \cite{DHK21}, who conjectured its semiclassical limit. They identified the quiver $Q_n$ and demonstrated its consistency with various numerical requirements, in addition to showing that for small values of~$n$ the claim follows from computations already in the literature (see \cite[Sec.~8]{DHK21}). They emphasized in particular that this description identifies the symplectic dual of $\ol{T^*(SL_n/U)}$, namely it is the Higgs branch $\cM_H(Q_n)$ of the relevant quiver gauge theory. We also note that the above is just one of several remarkable proposals in \cite{DHK21}, the others of which we do not address. 

Beyond confirming this proposal in quantized form, a further goal of ours is to formulate and prove a generalization to arbitrary unipotent reductions of $T^* SL_n$. Recall that the action of $U$ on $SL_n$ induces a moment map $T^* SL_n \to \fu^*$, and given a moment map value $\psi \in \fu^*$ we have the Hamiltonian reduction 
$$ T^* SL_n \sslash_\psi U := \Spec\, \C[(T^*SL_n \times_{\fu^*} {\psi}) / U]. $$
We have $T^* SL_n \sslash_0 U \cong \ol{T^*(SL_n/U)}$, hence as we vary $\psi$ we obtain a family of deformations of $\ol{T^*(SL_n/U)}$. 
	
Up to symplectic algebraic isomorphism, only finitely many isomorphism classes appear in this family. Specifically, writing $e_1, \dotsc, e_{n-1}$ for the Chevalley generators of $\fu$, any function $\{e_1, \dotsc, e_{n-1}\} \to \{1,0\}$ extends uniquely to a Lie algebra character $\psi \in \fu^*$, and any reduction of $T^* SL_n$ by $U$ is isomorphic to one of this form. We can in turn identify such functions with the set of ordered tuples $\vec{n} = (n_1,\dotsc,n_k)$ of positive integers summing to $n$: given $\vec{n}$ we set $\psi(e_i) = 0$ if $i = n_1 + \cdots + n_j$ for some $j$, and $\psi(e_i) = 1$ otherwise. 

\begin{Theorem}\label{thm:levithmintro}
Let $\vec{n} = (n_1,\dotsc,n_k)$ be an ordered tuple of positive integers summing to $n$, and let $\psi \in \fu^*$ be the associated character.  Then there exists an isomorphism $T^* SL_n \sslash_{\psi} U \cong \cM_C(Q_{\vec{n}})$ of algebraic varieties, where $\cM_C(Q_{\vec{n}})$ is the Coulomb branch of the $3d$ $\cN=4$ quiver gauge theory associated to the quiver
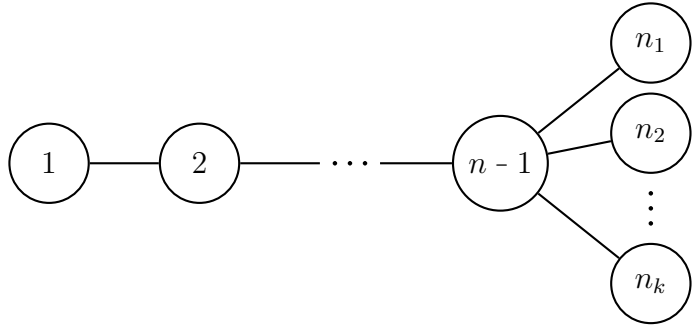
\begin{equation*} 
\begin{tikzpicture}[baseline=(current bounding box.center), thick, >=Stealth]
	\node[circle, draw, minimum size=30] (1) at (0,0) {$1$};
	\node[circle, draw, minimum size=30] (2) at (2,0) {$2$};
	\fill (3.8,0) circle (.03);
	\fill (4,0) circle (.03);
	\fill (4.2,0) circle (.03);
	\fill (8,-0.4) circle (.03);
	\fill (8,-0.6) circle (.03);
	\fill (8,-0.8) circle (.03);
	\node[circle, draw, minimum size=30] (n-1) at (6,0) {$n - 1$};
	\node [circle, draw, minimum size=30] (v1) at (8,1.6) {$n_1$};
	\node [circle, draw, minimum size=30] (v2) at (8,0.4) {$n_2$};
	\node [circle, draw, minimum size=30] (v4) at (8,-1.6) {$n_k$};
	
	\draw (2) -- (1);
	\draw (2) -- (3.6,0);
	\draw (n-1) -- (4.4,0);
	\draw (n-1) -- (v1);
	\draw (n-1) -- (v2);
	\draw (n-1) -- (v4);
\end{tikzpicture}
\end{equation*} with gauge group $\tilde{L}$ (depending on $\vec{n}$) defined by (\ref{eq:Tgroupsquare}). 
\end{Theorem}

At the other end of the spectrum from $\vec{n} = (1,\cdots,1)$ we have $\vec{n} = (n)$, which is associated to the Whittaker character $\psi_W$. That the Coulomb branch of the associated quiver gauge theory is isomorphic to the Whittaker reduction of $T^* SL_n$ was explicitly shown in \cite[Sec. 7.1]{DG19} and implicitly in \cite[Sec. 5(ii)]{BFN19} (using the identification of the Whittaker reduction with the product of the group and the Kostant slice \cite[Lem. 3.2.3(iii)]{GK22}). Thus Theorem~\ref{thm:levithmintro} interpolates between these established results and the proposal of \cite{DHK21}. 

The proof of Theorem \ref{thm:mainthmintro} makes crucial use of two key results in the literature. The first is a result of Ginzburg-Riche \cite{GR15} which provides an interpretation of $D_\hbar(SL_n/U)$ in terms of the affine Grassmannian $\Gr_{PGL_n}$. Recall that geometric Satake identifies $PGL_{n,\cO}$-equivariant perverse sheaves on $\Gr_{PGL_n}$ with representations of $SL_n$. Letting $\cA_{reg}$ denote the sheaf associated to the regular representation $\C[SL_n]$, the results of Ginzburg-Riche imply that $D_\hbar(SL_n/U)$ can be recovered as the $T_{\cO} \rtimes \bC^\times$-equivariant  cohomology of the !-restriction of $\cA_{reg}$ to $\Gr_{T}$, where $T$ is the maximal torus of $PGL_n$. 

On the other hand, Braverman-Finkelberg-Nakajima have shown that $\cA_{reg}$ is the pushforward to $\Gr_{PGL_n}$ of the dualizing sheaf of the space $\cR_{G,N}$ associated to the quiver $Q_{\vec{n}}$ when $\vec{n} = (n)$ \cite[Thm. 2.11]{BFN19}. Given these results, we derive Theorem \ref{thm:mainthmintro} using a certain compatibility between base change and ring objects in the derived Satake categories of different groups, established in Section \ref{sec:basechange}. This analysis is a variation on that of \cite[Sec. 2(vi)]{BFN19}, where the !-restriction of $\cA_{reg}$ to the identity (known by \cite[Thm. 7.3.1]{ABG04} to give the coordinate ring of the nilpotent cone) is related to the quiver gauge theory obtained from $Q_{(n)}$ by framing rather than gauging its rank $n$ vertex (see Remark \ref{rem:BFNnilconecomparison}). 

We use the same strategy to prove Theorem \ref{thm:levithmintro}, except that in place of \cite{GR15} we use recent results of Macerato \cite{Mac23}. These parallel those of \cite{GR15} except the role of $T$ is played by a general Levi subgroup of $PGL_n$. Note that such subgroups are also indexed by ordered tuples summing to $n$, and that the corresponding Levi is a quotient of the gauge group of the theory appearing in Theorem \ref{thm:levithmintro}. The effect of loop equivariance is not analyzed explicitly in \cite{Mac23}, which is why Theorem~\ref{thm:levithmintro} is only stated at the classical level.

More generally, it seems that Theorems \ref{thm:mainthmintro}, \ref{thm:Weylaction}, and \ref{thm:levithmintro} should admit extensions where $SL_n$ is replaced by $SO_{2n}$ or $Sp_n$. That is, assuming the generalization of \cite[Thm. 2.11]{BFN19} anticipated in \cite[Rem. 5.3]{BFN19}, such extensions should follow from a variation of the same argument. 

A consequence of Theorem \ref{thm:mainthmintro} is that it induces a canonical integrable system on $\affineClosureOfCotangentBundleofBasicAffineSpaceFORSLN$. More precisely, letting $\wG$ denote the gauge group associated to (\ref{eq:quiverQn}) and letting $\wW := N_{\wG}(\wT)/\wT$ denote its Weyl group, we have the following corollary of \cite[\textsection 5(v)]{BFN18}. 

\begin{Corollary}\label{As a Coulomb Branch affineClosureOfCotangentBundleofBasicAffineSpaceFORSLN Admits an Integrable System and is Birational To the Appropriate Space}
	There is an flat, injective ring homomorphism $\bC[\wft]^{\wW} \to \bC[\affineClosureOfCotangentBundleofBasicAffineSpaceFORSLN]$ whose image is a Poisson-commutative subalgebra. There is a birational map \begin{equation*}
		\affineClosureOfCotangentBundleofBasicAffineSpaceFORSLN \dashrightarrow{} T^*(\wT^{\vee})\sslash \wW = (\wT^{\vee} \times \wft)\sslash \wW\end{equation*} of varieties over $\wft\sslash \wW$, and in particular the generic fiber of  $\affineClosureOfCotangentBundleofBasicAffineSpaceFORSLN \to \wft\sslash \wW$ is the dual torus $\wT^{\vee}$. 
\end{Corollary}

In fact, the arguments of \cite[\textsection 5(v)]{BFN18}, combined with \cref{thm:mainthmintro}, give rise to a commutative subalgebra \begin{equation}\label{ComAlg}\bC_{\hbar}[\wft]^{\wW} = \bC[\wft]^{\wW}[\hbar] \subseteq D_\hbar(SL_n/U).\end{equation} We expect that this subalgebra can be obtained from the \textit{Gelfand-Tsetlin subalgebra} of the asymptotic universal enveloping algebra of $\mathfrak{gl}_n$; we make this expectation precise in \cref{Bonus Expectations of Theorem}. 

We also recall that by a recent result of \cite{BellamyCoulombBranchesHaveSymplecticSingularities} the Coulomb branch of any $3d\text{ }\mathcal{N}=4$ gauge theory admits symplectic singularities in the sense of \cite{BeauvilleSymplecticSingularities}. Combined this with, Theorem~\ref{thm:mainthmintro} thus gives a new proof for $SL_n$ of the Ginzburg-Kazhdan conjecture that  $\affineClosureOfCotangentBundleofBasicAffineSpace$ has symplectic singularities for any semisimple $G$ (this was first proved in \cite{JiaTheGeometryOfTheAffineClosureOfCotangentBundleOfBasicAffineSpaceForSLn} for $SL_n$ and in \cite{GannonAProofofGinzburgKazhdanConjecture} for general $G$). 

\subsection{Acknowledgements} We thank Ana B\u alibanu, David Ben-Zvi, Sabin Cautis, Tudor Dimofte, Nikolay Grantcharov, Victor Ginzburg, Justin Hilburn, Mark Macerato, Hiraku Nakajima, and Sam Raskin for useful comments and discussions. H. W. was supported by NSF CAREER grant DMS-2143922. 

\section{Generalities and conventions} \label{sec:conventions}

\subsection{Coulomb branches}\label{sec:liecon} 

Let $G$ be a complex reductive group. We write $G_\cK$ for its formal loop group, $G_\cO \subset G_\cK$ for the subgroup of loops extending over the formal disc, and $\hG_\cO := G_\cO \rtimes \C^\times$ and $\hG_\cK := G_\cK \rtimes \C^\times$ for their loop-equivariant extensions. We have the affine Grassmannian $\Gr_G := G_\cK/G_\cO \cong \hG_\cK/ \hG_\cO$, and the (loop) equivariant Grassmannian ${}_G \Gr_G := \hG_\cO \backslash \hG_\cK/ \hG_\cO$ (we take the stack quotient in the \'etale topology).  If $L \subset G$ is a subgroup, we also write ${}_L \Gr_G := \hL_\cO \backslash \hG_\cK/ \hG_\cO$.

If $N$ is a finite-dimensional representation of $G$, the associated Coulomb branch $\cM_{G,N}$ is defined as follows \cite{BFN18}. We write $\cT_{G,N} := G_\cK \times_{G_\cO} N_\cO$ and consider the Cartesian diagram
\begin{equation*}
	\begin{tikzpicture}
		[baseline=(current  bounding  box.center),thick,>=\arrtip]
		\node (a) at (0,0) {$\cR_{G,N}$};
		\node (b) at (3,0) {$\Gr_G \times N_\cO$};
		\node (c) at (0,-1.5) {$\cT_{G,N}$};
		\node (d) at (3,-1.5) {$\Gr_G \times N_\cK,$};
		\draw[->] (a) to node[above] {$ $} (b);
		\draw[->] (b) to node[right] {$ $} (d);
		\draw[->] (a) to node[left] {$ $}(c);
		\draw[->] (c) to node[above] {$ $} (d);
	\end{tikzpicture}
\end{equation*}
where the lower map takes $[g, s]$ to $([g],[gs])$ for $g \in G_\cK$ and $s \in N_\cO$. 
The fiber product $\cR_{G,N}$ is an ind-scheme which is not of ind-finite type, and which carries an induced action of $\hG_\cO$. The quantized Coulomb branch algebra $\bC_\hbar[\cM_{G,N}]$ is the equivariant Borel-Moore homology $H_\bullet^{\hG_\cO}(\cR_{G,N})$, equipped with a certain convolution product. The fact that $\cR_{G,N}$ is not of ind-finite type means that defining $H_\bullet^{\hG_\cO}(\cR_{G,N})$ requires specifying certain degree normalizations on finite-type approximations, see \cite[Sec. 2(ii)]{BFN18}. The parameter $\hbar$ is by definition the equivariant parameter of the loop rotation $\bC^\times$. The ring $\bC_\hbar[\cM_{G,N}]$ becomes commutative by setting it to zero, and the Coulomb branch $\cM_{G,N}$ is the resulting affine variety. 

\subsection{Equivariant sheaves}

Given a complex variety $X$, we write $D(X)$ for the ind-completion of the bounded derived category of constructible sheaves on $X$. If an algebraic group $G$ acts on $X$, we write $D_G(X)$ for the corresponding equivariant category. As in \cite{BFN19}, the foundations of \cite{BernsteinLuntzEquivariantSheavesandFunctors} are sufficient for our purposes, up to standard adaptations allowing $X$ and $G$ to be infinite-dimensional (see \cite[Sec. 2(i)]{BFN19} or \cite[Sec. 9.1]{Ach21}). 


We do find, however, that our arguments are more transparent when written systematically in terms of quotient stacks, and we freely use $D(X/H)$ as a synonym for $D_H(X)$. Our notation is set up so that the derived Satake category, written as $D_G(\Gr_G)$ in \cite{BFN19}, is instead written as $D({}_G \Gr_G)$ here (more precisely, the latter is the loop-equivariant version). 

This convention lets us use uniform notation for ordinary pushforwards/pullbacks and the generalized pushforwards/pullbacks of \cite[Sec. 6]{BernsteinLuntzEquivariantSheavesandFunctors}, including induction/restriction functors. That is, if $H \to G$ is a homomorphism, $f: Y \to X$ an $H$-equivariant map, and $\phi: Y/H \to X/G$ the induced map of quotient stacks, then the general pushforward $Q_{f*}: D_H(Y) \to D_G(X)$ of \cite[Sec. 6]{BernsteinLuntzEquivariantSheavesandFunctors} is synonymous with $\phi_*$, similarly for $Q_f^*$ and $\phi^*$. 

A related terminological convention is that we use the term smooth base change to refer also to \cite[Thm. 3.4.1, Prop. 7.2]{BernsteinLuntzEquivariantSheavesandFunctors}. These state that $Q_f^*$ commutes with the other sheaf operations when $X = Y$ and $f = \id$, and may be interpreted as smooth base change for the quotient map $\phi$ (more precisely, pro-smooth base change if $G$ and $H$ are only pro-algebraic).



\subsection{$!$-sheaves and $*$-sheaves}
We now explain how our discussion can be interpreted from the foundational perspective of \cite{Ras14b,BKV22}. These are independent of \cite{BernsteinLuntzEquivariantSheavesandFunctors} and involve stacks in a more fundamental way. This discussion will not be logically essential for us, and we emphasize again that the reader familiar with \cite{BFN19} and its technical language may regard our use of stacks as a notational device. However, some conceptual points are easier to discuss from this perspective. In particular, making the distinction between $!$-sheaves and $*$-sheaves as in \cite{Ras14b} provides a convenient way of sidestepping the infinite shifts which appear in \cite{BFN18,BFN19} (e.g. ``$\omega_{N_\cO} \cong \bC_{N_\cO}[2 \dim N_\cO]$''). 

This distinction appears specifically when trying to extend classical constructible (or $\ell$-adic, de Rham, etc.) sheaf theory to accommodate infinite-type spaces or groups. Roughly, one must choose between having canonically functorial $!$-pullbacks but $*$-pushforwards which are undefined or depend on non-canonical choices ($!$-sheaves), or having the converse ($*$-sheaves). In following \cite{BFN19} we implicitly work with $*$-sheaves (or a variant thereof), though \cite{BKV22} and many other references prioritize $!$-sheaves. 


Before recalling these notions more precisely we fix some notation. Write $\Catinfty$ for the $\infty$-category of small $\infty$-categories and $\PrL$ for that of presentable $\infty$-categories and colimit-preserving functors (the descent constructions below necessitate enhanced catgories). We work over $\bC$, and write $\Sch_{ft}$ for the category of finite-type schemes, $\Sch$ for the category of qcqs schemes, and $\Stk$ for the $\infty$-category of \'etale stacks, which we will hereafter simply refer to as the category of stacks.

For the present purposes, we say a stack is geometric if it admits an affine \'etale surjection $\Spec A \to X$ which is strongly pro-smooth (see \cite[Not. 1.1.1(d)]{BKV22}). A stack is ind-geometric if it is a filtered colimit of geometric stacks along finitely presented closed immersions. The \'etale quotient $\cR_{G,N}/G_\cO$ is an ind-geometric stack in this sense (we borrow the terminology of \cite{CWig} but adapt some details to the constructible setting). We write $\GStk$ and $\indGStk$ for the categories of geometric and ind-geometric stacks. 

From classical sheaf theory we take as given the functors $D^!_c, D^*_c: \Sch_{ft}^{\op} \to \Catinfty$ that take a morphism $f:X \to Y$ respectively to the pullback $f^!: D_c(Y) \to D_c(X)$ or $f^*: D_c(Y) \to D_c(X)$ (here $D_c(X) \subseteq D(X)$ denotes compact objects). Starting from $D^!_c$ we now
\begin{enumerate}
	\item left Kan extend to a functor $D^!_c: \Sch^{\op} \to \Catinfty$,
	\item right Kan extend to a functor $D^!_c: \Stk^{\op} \to \Catinfty$,
	\item ind-complete to a functor $D^!: \Stk^{\op} \to \Pr^L$. 
\end{enumerate}
The value $D^!(Z)$ of this last functor on an arbitrary stack $Z$ is the category of ind-constructible $!$-sheaves on $Z$, specifically its renormalized version in the sense of \cite[Sec. 12]{AG15}. 

The category of $*$-sheaves may be defined in a dual fashion \cite[Sec. 2.5]{Ras14b}, but here we describe a variant adapted to our needs. Starting from $D^*_c$ we 
\begin{enumerate}
	\item left Kan extend to a functor $D_{c}^*: \Sch^{\op} \to \Catinfty$,
	\item right Kan extend to a functor $D_{c}^*: \GStk^{\op} \to \Catinfty$,
	\item ind-complete to a functor $D^*: \GStk^{\op} \to \Pr^L$,
	\item pass to right adjoints to obtain a functor $D_*: \GStk \to \Pr^L$,
	\item left Kan extend to a functor $D_*: \indGStk \to \PrL$. 
\end{enumerate}
Note that the right adjoints in (4) are still valued in $\Pr^L$ since those in (3) preserve compact objects. 

Write $D(Z)$ for the category assigned by (5) to an ind-geometric stack $Z$, and observe that if $Z$ is a geometric stack then $D(Z)$ is simply the category of $*$-sheaves on $Z$. Moreover, arguments adapted straightforwardly from \cite{Ras14b,BKV22} show that the categories $D(Z)$ and the associated pushforwards $f_*: D(Z') \to D(Z)$ have all the familiar sheaf theoretic properties needed for the constructions in \cite{BFN19} and the body of this paper. In particular, (1) if $f$ is ind-proper then $f_*$ admits a right adjoint $f^!$, (2) if $f$ is pro-smooth then $f_*$ has a left adjoint $f^*$, and (3) in the latter case $f^*$ satisfies base change against the other sheaf operations (and thus the results \cite[Thm. 3.4.1, Prop. 7.2]{BernsteinLuntzEquivariantSheavesandFunctors} hold in this setting). 

\subsection{Placidity} While $D^!(Z)$ and $D(Z)$ need not be equivalent in general, they are when $Z$ is placid, as in \cite[Sec. 4]{Ras14b} or \cite[Sec. 1.3]{BKV22}. When $Z$ is a qcqs scheme, placidity is the condition that it can be written as an inverse limit of finite-type schemes along smooth maps. Notably for us, all spaces relevant to Coulomb branches are placid; for example, the placidity of $\cR_{G,N}$ follows from the discussion of \cite[Sec. 2(i)]{BFN18}. 

More precisely, for placid $Z$ there is an equivalence $D^!(Z) \cong D(Z)$ which is canonical up to a discrete choice. The possible choices are distinguished by where $\omega_Z := p^!(\bC_\pt) \in D^!(Z)$ is sent (here $p: Z \to \pt)$. Its image $\omega_Z^{ren} \in D(Z)$ is referred to as the associated renormalized dualizing complex; note that in general $D(Z)$ does not admit a good notion of dualizing complex, since the functor $p^!$ is not a priori defined for $*$-sheaves. Any placid \emph{scheme} has a canonical renormalized dualizing complex, which in the case of $N_\cO$ is $\omega^{ren}_{N_\cO} \cong \bC_{N_\cO}$. It is this relationship which following \cite{BFN18} we would informally write as $\omega_{N_\cO} \cong \bC_{N_\cO}[2 \dim N_\cO]$. 

\subsection{Coulomb branches and ring objects}
Fixing again a reductive group $G$ and representation $N$, we recall the description of Coulomb branches via ring objects in the derived Satake category \cite{BFN19}. We first fix the renormalized dualizing complex $\omega_{G,N}^{ren} := q^!(\C_{N_\cO/\hG_\cO})\in D_{\hG_\cO}(\cR_{G,N})$, where $q$ denotes the (ind-proper) projection $\cR_{G,N} \to \Gr_G \times N_\cO \to N_\cO$ (and its quotient by $\hG_\cO$). The degree normalizations used in \cite[Sec. 2(ii)]{BFN18} to define $H_\bullet^{\hG_\cO}(\cR_{G,N})$ via finite-type approximations are encoded in this choice of $\omega^{ren}_{G,N}$, in the sense that we indeed recover $H_\bullet^{\hG_\cO}(\cR_{G,N})$ as the cohomology of $\omega_{G,N}^{ren}$. 

The pushforward $\cA_{G,N}:= \pi_*(\omega_{G,N}^{ren}) \in D({}_G \Gr_G)$ along the projection $\pi: \cR_{G,N} \to \Gr_G$ has the structure of a ring object with respect to  convolution in $D({}_G \Gr_G)$ \cite[Prop. 2.1]{BFN19}. Further pushing forward to a point of course recovers $H_\bullet^{\hG_\cO}(\cR_{G,N})$, but one of the key ideas of \cite{BFN19} is that ring objects in $D({}_G \Gr_G)$ provide a general setting for defining and studying Coulomb branches that do not arise from the construction of \cite{BFN18}. 


\section{Convolution algebras and base change}\label{sec:basechange}

As mentioned in the introduction, to prove our main theorems we will use a certain base change result for ring objects in the derived Satake categories of different groups, formulated as Proposition~\ref{prop:algiso} below. This is to be expected given related results in the literature (e.g. it is a relative of both \cite[Lem. 2.9]{BFN19} and \cite[Lem. 4.2.1]{BDFRT22}), but we include a complete proof in the absence of an explicit reference. 

We begin by fixing a Cartesian diagram of reductive groups as follows.
\begin{equation}\label{eq:groups}
	\begin{tikzpicture}[baseline=(current  bounding  box.center),thick,>=\arrtip]
		\newcommand*{\ha}{3}; \newcommand*{\va}{-1.5}; 
		
		\node (aa) at (0,0) {$\wL$};
		\node (ab) at (\ha,0) {$\wG$};
		\node (ba) at (0,\va) {$L$};
		\node (bb) at (\ha,\va) {$G$};
		
		\draw[->] (aa) to node[above,pos=.5] {$ $} (ab);
		\draw[->] (ba) to node[above,pos=.5] {$ $} (bb);
		\draw[->] (aa) to node[right,pos=.5] {$ $} (ba);
		\draw[->] (ab) to node[right,pos=.5] {$ $} (bb);
	\end{tikzpicture}
\end{equation}
We assume that the right arrow is a quotient map, hence so is the left arrow. In the next section $L$ will be a Levi subgroup of $G$, but for now we will not use this. 

The diagram (\ref{eq:groups}) induces a corresponding diagram of loop-equivariant affine Grassmannians. We factor this as follows, using the notation of Section \ref{sec:liecon}. 
\begin{equation}\label{eq:equivgrass}
	\begin{tikzpicture}[baseline=(current  bounding  box.center),thick,>=\arrtip]
		\newcommand*{\ha}{3}; \newcommand*{\hb}{3};
		\newcommand*{\va}{-1.5}; \newcommand*{\vb}{-1.5}; 
		
		\node (aa) at (0,0) {${}_{\wL}\Gr_{\wL}$};
		\node (ab) at (\ha,0) {${}_{\wL}\Gr_{\wG}$};
		\node (ac) at (\ha+\hb,0) {${}_{\wG}\Gr_{\wG}$};
		
		\node (ba) at (0,\va) {${}_{L}\Gr_{L}$};
		\node (bb) at (\ha,\va) {${}_{L}\Gr_{G}$};
		\node (bc) at (\ha+\hb,\va) {${}_{G}\Gr_{G}$};
		
		\draw[->] (aa) to node[above,pos=.5] {$i' $} (ab);
		\draw[->] (ab) to node[above,pos=.5] {$j' $} (ac);
		\draw[->] (ba) to node[above,pos=.5] {$i $} (bb);
		\draw[->] (bb) to node[above,pos=.5] {$j $} (bc);
		
		\draw[->] (aa) to node[right,pos=.5] {$p'' $} (ba);
		\draw[->] (ab) to node[right,pos=.5] {$p' $} (bb);
		\draw[->] (ac) to node[right,pos=.5] {$p $} (bc);
	\end{tikzpicture}
\end{equation}
In particular, the forgetful functor from $\hG_\cO$-equivariant constructible sheaves on $\Gr_G$ to $\hL_\cO$-equivariant sheaves will be denoted by $j^*$. 

As remarked in \cite[Sec. 5(xi)]{BFN19}, if $\cA \in D({}_{\wG}\Gr_{\wG})$ has the structure of a ring object with respect to convolution, its pushforward $p_{*}(\cA) \in D({}_{G}\Gr_{G})$ naturally inherits a ring structure, as does $i'^! j'^*(\cA) \in D({}_{\wL}\Gr_{\wL})$. The desired claim is that we obtain the same ring object in $D({}_{L}\Gr_{L})$ if we apply both constructions in either of the two possible orders. 

\begin{Proposition}\label{prop:algiso}
	Given a ring object $\cA \in D({}_{\wG}\Gr_{\wG})$, there is a canonical isomorphism \begin{equation}\label{eq:algebraiso}
		 p''_{*} i'^! j'^*(\cA) \cong i^!j^*p_{*}(\cA)
	\end{equation}
	of ring objects in in  $D({}_{L}\Gr_{L})$. 
\end{Proposition}

Note that $i^! j^*$ can be interpreted as a renormalized !-pullback $(ji)^!_{ren}$ in the sense of \cite[Sec. 4.9]{Ras14b}. However, the claim is ultimately an identity between two natural transformations, and to define and analyze these transformations we must express them in terms of adjunctions involving ind-proper !-pullback and smooth *-pullback. We discuss this point some more at the end of the section. We also clarify again that when discussing ring objects we only consider them at the classical level, i.e. we prove the above statement interpreting $D({}_{L}\Gr_{L})$ as a triangulated category rather than its underlying stable $\infty$-category. 

Before proving Proposition \ref{prop:algiso} we explain the two ring object constructions explicitly. They are induced by a pair of natural transformations
\begin{equation}\label{eq:laxmaps}
	p_{*}(-) \conv p_{*}(-) \to p_{*}(- \conv - ), \quad i^! j^*(-) \conv i^! j^*(-) \to i^! j^*(- \conv - ). 
\end{equation}
Given these, a multiplication $m: \cA \conv \cA \to \cA$ induces a multiplication on $p_{*}(\cA)$ via
$$ p_{*}(\cA) \conv p_{*}(\cA) \to p_{*}(\cA \conv \cA) \xrightarrow{p_{*}(m)} p_{*}(\cA), $$
similarly for $i^! j^*p_{*}(\cA)$. For details on the associativity of these multiplications we refer to \cite[Sec. 2.2]{Mac23}, these not being directly relevant to the proof of Proposition \ref{prop:algiso}. 

To define the left transformation in (\ref{eq:laxmaps}), consider the following diagram.
\begin{equation}\label{eq:pconvsquare}
\begin{tikzpicture}[baseline=(current  bounding  box.center),thick,>=\arrtip]
	\newcommand*{\ha}{3.4}; 
	\newcommand*{\va}{-1.5}; \newcommand*{\vb}{-1.5}; 
	
	\node (aa) at (0,0) {${}_{\wG}\Gr_{\wG} \tighttimes {}_{\wG}\Gr_{\wG}$};
	\node (ab) at (\ha,0) {${}_{G}\Gr_{G} \tighttimes {}_{G}\Gr_{G}$};
	\node (ba) at (0,\va) {${}_{\wG}\Gr_{\wG}\Gr_{\wG}$};
	\node (bb) at (\ha,\va) {${}_{G}\Gr_{G}\Gr_{G}$};
	\node (ca) at (0,\va+\vb) {${}_{\wG}\Gr_{\wG}$};
	\node (cb) at (\ha,\va+\vb) {${}_{G}\Gr_{G}$};
	
	\draw[->] (aa) to node[above,pos=.5] {$p \tighttimes p $} (ab);
	\draw[->] (ba) to node[above,pos=.5] {$q $} (bb);
	\draw[->] (ca) to node[above,pos=.5] {$p $} (cb);
	
	\draw[<-] (aa) to node[left,pos=.5] {$d_{\wG} $} (ba);
	\draw[<-] (ab) to node[right,pos=.5] {$d_G $} (bb);
	\draw[->] (ba) to node[left,pos=.5] {$m_{\wG} $} (ca);
	\draw[->] (bb) to node[right,pos=.5] {$m_G $} (cb);
\end{tikzpicture}
\end{equation}
Here and below ${}_{G}\Gr_{G}\Gr_{G} := \hG_\cO \backslash \hG_\cK \times_{\hG_\cO} \hG_\cK/\hG_\cO$, etc., denote the loop-equivariant convolution spaces and their variants. The vertical maps are those defining the convolution product (so that $\cF \conv \cG := m_{G*} d_G^*(\cF \boxtimes \cG)$ for $\cF, \cG \in D({}_G \Gr_G)$), and the map $q$ is induced by $\wG \to G$. By adjunction the upper square induces a natural transformation
$$ m_{G*} d^*_G (p \tighttimes p)_* \to m_{G*} q_* d^*_{\wG} \cong p_* m_{\wG*} d^*_{\wG}, $$
yielding the desired transformation
$$ p_{*}(-) \conv p_{*}(-) \cong m_{G*} d^*_G (p \tighttimes p)_*(- \boxtimes -)  \to p_* m_{\wG*} d^*_{\wG} (- \boxtimes -) \cong p_{*}(- \conv - ). $$ 

To define the right transformation in (\ref{eq:laxmaps}), we consider the following diagram.
\begin{equation}\label{eq:ijconvsquare}
	\begin{tikzpicture}[baseline=(current  bounding  box.center),thick,>=\arrtip]
		\newcommand*{\ha}{3.4}; \newcommand*{\hb}{3.3};
		\newcommand*{\va}{-1.5}; \newcommand*{\vb}{-1.5}; 
		
		\node (aa) at (0,0) {${}_{L}\Gr_{L} \tighttimes {}_{L}\Gr_{L}$};
		\node (ab) at (\ha,0) {${}_{L}\Gr_{G} \tighttimes {}_{G}\Gr_{G}$};
		\node (ac) at (\ha+\hb,0) {${}_{G}\Gr_{G} \tighttimes {}_{G}\Gr_{G}$};
		
		\node (ba) at (0,\va) {${}_{L}\Gr_{L}\Gr_{L}$};
		\node (bb) at (\ha,\va) {${}_{L}\Gr_{G}\Gr_{G}$};
		\node (bc) at (\ha+\hb,\va) {${}_{G}\Gr_{G}\Gr_{G}$};
		
		\node (ca) at (0,\va+\vb) {${}_{L}\Gr_{L}$};
		\node (cb) at (\ha,\va+\vb) {${}_{L}\Gr_{G}$};
		\node (cc) at (\ha+\hb,\va+\vb) {${}_{G}\Gr_{G}$};
		
		\draw[->] (aa) to node[above,pos=.5] {$i \tighttimes ji $} (ab);
		\draw[->] (ab) to node[above,pos=.5] {$j \tighttimes \id $} (ac);
		\draw[->] (ba) to node[above,pos=.5] {$k $} (bb);
		\draw[->] (bb) to node[above,pos=.5] {$\ell $} (bc);
		\draw[->] (ca) to node[above,pos=.5] {$i $} (cb);
		\draw[->] (cb) to node[above,pos=.5] {$j $} (cc);
		
		\draw[<-] (aa) to node[right,pos=.5] {$d_L $} (ba);
		\draw[<-] (ab) to node[right,pos=.5] {$d'_G $} (bb);
		\draw[<-] (ac) to node[right,pos=.5] {$d_G $} (bc);
		\draw[->] (ba) to node[right,pos=.5] {$m_L $} (ca);
		\draw[->] (bb) to node[right,pos=.5] {$m'_G $} (cb);
		\draw[->] (bc) to node[right,pos=.5] {$m_G $} (cc);
	\end{tikzpicture}
\end{equation}
We first note there is a canonical isomorphism 
\begin{equation}\label{eq:shriekstariso}
	d^*_L (i \times i)^! (\id \times j)^* \cong k^! d'^*_G.
\end{equation}
This can be identified with the isomorphism between the two renormalized !-pullbacks around the top left square in (\ref{eq:ijconvsquare}), but in the argument below we will use the fact that it can be expressed in terms of smooth base change as follows. We factor the upper left square as 
\begin{equation}\label{eq:antidiagcube}
	\begin{tikzpicture}[baseline=(current  bounding  box.center),thick,>=\arrtip]
		\newcommand*{\ha}{3.3}; \newcommand*{\hb}{3.3}; \newcommand*{\hc}{3.3};
		\newcommand*{\va}{-1.5}; \newcommand*{\vb}{-1.5}; 
		
		\node (aa) at (0,0) {${}_{L}\Gr_{L} \tighttimes {}_{L}\Gr_{L}$};
		\node (ab) at (\ha,0) {${}_{L}\Gr_{G} \tighttimes {}_{L}\Gr_{G}$};
		\node (ac) at (\ha+\hb,0) {$ $};
		\node (ad) at (\ha+\hb+\hc,0) {${}_{L}\Gr_{G} \tighttimes {}_{G}\Gr_{G}$};
		
		\node (ba) at (0,\va) {$ $};
		\node (bb) at (\ha,\va) {${}_{L}\Gr_{L} \tighttimes {}_{L}\Gr_{G}$};
		\node (bc) at (\ha+\hb,\va) {${}_{L}\Gr_{L} \tighttimes {}_{G}\Gr_{G}$};
		\node (bd) at (\ha+\hb+\hc,\va) {$ $};
		
		\node (ca) at (0,\va+\vb) {${}_{L}\Gr_{L}\Gr_{L}$};
		\node (cb) at (\ha,\va+\vb) {${}_{L}\Gr_{L}\Gr_{G}$};
		\node (cc) at (\ha+\hb,\va+\vb) {$ $};
		\node (cd) at (\ha+\hb+\hc,\va+\vb) {${}_{L}\Gr_{G}\Gr_{G}.$};
		
		\draw[->] (aa) to node[above,pos=.5] {$i \tighttimes i $} (ab);
		\draw[->] (ab) to node[above,pos=.5] {$\id \tighttimes j $} (ad);
		\draw[->] (bb) to node[above,pos=.5] {$\id \tighttimes j $} (bc);
		\draw[->] (ca) to node[above,pos=.5] {$k_a $} (cb);
		\draw[->] (cb) to node[above,pos=.5] {$k_b $} (cd);
		
		\draw[<-] (aa) to node[left,pos=.5] {$d_L $} (ca);
		\draw[<-] (ab) to node[right,pos=.5] {$i \tighttimes \id $} (bb);
		\draw[<-] (bb) to node[right,pos=.5] {$d''_G $} (cb);
		\draw[<-] (ad) to node[right,pos=.5] {$d'_G $} (cd);

		\draw[->] (aa) to node[below left,pos=.5] {$\id \tighttimes i $} (bb);
		\draw[->] (bc) to node[below right,pos=.5] {$i \tighttimes \id $} (ad);
	\end{tikzpicture}
\end{equation}
The two quadrilateral faces of this diagram are Cartesian, and the pentagonal face becomes a Cartesian square after composing $\id \tighttimes j$ and $d''_G$. We now obtain (\ref{eq:shriekstariso}) as the composition
\begin{equation}\label{eq:threesbc} 
	\begin{aligned} 
		d_L^* (i \tighttimes i)^! (\id \tighttimes j)^* &\cong d_L^* (\id \times i)^! (i \times \id)^! (\id \tighttimes j)^* \\
	&\cong k^!_a d''^*_G (i \tighttimes \id)^! (\id \tighttimes j)^* \\
	&\cong k^!_a d''^*_G (\id \tighttimes j)^* (i \tighttimes \id)^!\\ &\cong k^!_a k^!_b d'^*_G, 
	\end{aligned}
\end{equation} 
where the first isomorphism is trivial and the rest arise from smooth base change. We then have a composition
$$ m_{L*} d^*_L (i \tighttimes i)^! (j \tighttimes j)^* \cong m_{L*} k^! d'^*_G (j \tighttimes \id)^* \to i^! m'_{G*} d'^*_G (j \tighttimes \id)^* \cong i^! j^* m_{G*} d^*_G, $$
where the first factor uses (\ref{eq:shriekstariso}), the second uses the base change map associated to the bottom left square in~(\ref{eq:ijconvsquare}), and the third uses the smooth base change isomorphism associated to the bottom right square (i.e. that forgetting equivariance commutes with pushforward). This now yields the desired transformation
$$ 
i^! j^*(-) \conv i^! j^*(-) \cong m_{L*} d^*_L (i \tighttimes i)^! (j \tighttimes j)^*(- \boxtimes -) \to i^! j^* m_{G*} d^*_G(- \boxtimes -) \cong i^! j^*(- \conv - ). $$ 

\begin{proof}[Proof of Proposition \ref{prop:algiso}]
We first note that both squares in (\ref{eq:equivgrass}) are Cartesian. Indeed, the diagram of classifying stacks associated to (\ref{eq:groups}) is Cartesian (i.e. $B\wL \cong BL \times_{BG} B\wG$) since $\wG \to G$ is a quotient map, hence $\wL\backslash \wG \cong L\backslash G$, hence the natural map factors as 
$$  B\wL \cong \wL\backslash \wG/\wG \cong L\backslash G/\wG \cong BL \times_{BG} B\wG. $$
It follows that $B\wL_\cO \cong BL_\cO \times_{BG_\cO} B\wG_\cO$ by taking mapping spaces from the formal disk $D$ (or repeating the argument), and that the outer square in (\ref{eq:equivgrass}) is Cartesian by taking mapping spaces from the formal bubble $D \cup_{D^\times} D$. But then $j'$ is the base change of $j$ since both are base changed from $BL_\cO \to BG_\cO$, hence the right square is Cartesian, hence so is the left. 

In particular, the isomorphism (\ref{eq:algebraiso}) may be defined as the composition 
$$ p''_{*} i'^! j'^*(\cA) \cong i^!p'_*j'^*(\cA) \cong i^!j^*p_*(\cA)$$
of proper and smooth base change isomorphisms. We want to show this identifies the induced multiplications on each side. That is, we wish to show it fits into a diagram 
	\begin{equation*}
		\begin{tikzpicture}[baseline=(current  bounding  box.center),thick,>=\arrtip]
			\newcommand*{\ha}{4.4}; \newcommand*{\hb}{3.4};
			\newcommand*{\va}{-1.5}; \newcommand*{\vb}{-1.5}; 
			
			\node (ba) at (0,\va) {$i^!j^*p_{*}(\cA) \conv i^!j^*p_{*}(\cA)$};
			\node (bb) at (\ha,\va) {$i^!j^*p_{*}(\cA \conv \cA)$};
			\node (bc) at (\ha+\hb,\va) {$i^!j^*p_{*}(\cA)$};
			
			\node (aa) at (0,0) {$p''_{*} i'^! j'^*(\cA) \conv p''_{*} i'^! j'^*(\cA)$};
			\node (ab) at (\ha,0) {$p''_{*} i'^! j'^*(\cA \conv \cA)$};
			\node (ac) at (\ha+\hb,0) {$p''_{*} i'^! j'^*(\cA)$};
			
			\draw[->] (aa) to node[above,pos=.5] {$ $} (ab);
			\draw[->] (ab) to node[above,pos=.5] {$ $} (ac);
			\draw[->] (ba) to node[above,pos=.5] {$ $} (bb);
			\draw[->] (bb) to node[above,pos=.5] {$ $} (bc);
			
			\draw[->] (aa) to node[right,pos=.5] {$ $} (ba);
			\draw[->] (ab) to node[right,pos=.5] {$ $} (bb);
			\draw[->] (ac) to node[right,pos=.5] {$ $} (bc);
		\end{tikzpicture}
	\end{equation*}
	in $D({}_L\Gr_L)$, where the horizontal arrows on the left are given by the constructions described before the proof. The right square is just functoriality, so it remains to construct the left. 
	
	The two paths around this square arise from a pair of natural transformations 
	$$m_{L*}d^*_L (p'' \tighttimes p'')_* (i' \tighttimes i')^! (j' \tighttimes j')^*  \to i^! j^* p_{*} m_{\wG*} d^*_{\wG}.$$ 
	Unwinding the constructions, these transformations are respectively described by the following string diagrams, read top to bottom (for compactness we write $p''_{12}$ for $p'' \tighttimes p''$, $j'_1$ for $j' \tighttimes \id$, etc.).  
	\begin{equation}\label{eq:stringdiags}
		\begin{tikzpicture}[baseline=(current  bounding  box.center),thick,>=\arrtip]
			\newcommand*{\ha}{1.2}; 
			\newcommand*{\va}{-1}; 
			\newcommand*{\vb}{0.4}; 
			
			\node[matrix] at (0,0) {
				
				\node (aa) at (0*\ha,\vb) {$m_{L*}$};
				\node (ab) at (1*\ha,\vb) {$d^*_{L}$};
				\node (ac) at (2*\ha,\vb) {$p''_{12*}$};
				\node (ad) at (3*\ha,\vb) {$i'^{!}_{12}j'^{*}_2$};
				\node (ae) at (4*\ha,\vb) {$j'^*_1$};
				
				\node (ba) at (0*\ha,5*\va-\vb) {$i^!$};
				\node (bb) at (1*\ha,5*\va-\vb) {$j^*$};
				\node (bc) at (2*\ha,5*\va-\vb) {$p_{*}$};
				\node (bd) at (3*\ha,5*\va-\vb) {$m_{\wG*}$};
				\node (be) at (4*\ha,5*\va-\vb) {$d^*_{\wG}$};
				
				\node at (.6*\ha, .75*\va) {$q''_* $};
				\node at (2.2*\ha, .95*\va) {$d^*_{\wL} $};
				\node at (1.0*\ha, 1.85*\va) {$m_{\wL*} $};
				\node at (2.0*\ha, 1.85*\va) {$k'^! $};
				\node at (3.45*\ha, 1.7*\va) {$d'^*_{\wG} $};
				\node at (-.35*\ha, 2.5*\va) {$p''_{*} $};
				\node at (.6*\ha, 2.70*\va) {$i'^! $};
				\node at (2.45*\ha, 2.7*\va) {$m'_{G*} $};
				\node at (1.0*\ha, 3.6*\va) {$p'_{*} $};
				\node at (1.9*\ha, 3.5*\va) {$j'^* $};
				\node at (3.25*\ha, 2.7*\va) {$\ell'^* $};
				
				\foreach \ma/\na/\mb/\nb/\curvefactor in {0/1/3/4/0.7, 3/1/0/4/0.7, 1/0/4/3/0.7, 4/2/1/5/0.7, 2/0/0/2/0.7, 0/3/2/5/0.7} {
					\draw ({\ma*\ha}, {\na*\va}) to[out=270,in=90,looseness=\curvefactor] ({\mb*\ha}, {\nb*\va});
					
				}
				
				\foreach \m/\na/\nb in {0/0/1, 0/2/3, 0/4/5, 3/0/1, 3/4/5, 4/0/2, 4/3/5} {
					\draw ({\m*\ha}, {\na*\va}) to ({\m*\ha}, {\nb*\va});
				}\\
			};
			
			\node[matrix] at (7,0) {
				
				\node (aa) at (0*\ha,\vb) {$m_{L*}$};
				\node (ab) at (1*\ha,\vb) {$d^*_{L}$};
				\node (ac) at (2*\ha,\vb) {$p''_{12*}$};
				\node (ad) at (3*\ha,\vb) {$i'^{!}_{12}j'^{*}_2$};
				\node (ae) at (4*\ha,\vb) {$j'^*_1$};
				
				\node (ba) at (0*\ha,5*\va-\vb) {$i^!$};
				\node (bb) at (1*\ha,5*\va-\vb) {$j^*$};
				\node (bc) at (2*\ha,5*\va-\vb) {$p_{*}$};
				\node (bd) at (3*\ha,5*\va-\vb) {$m_{\wG*}$};
				\node (be) at (4*\ha,5*\va-\vb) {$d^*_{\wG}$};
				
				\node at (4.4*\ha, 2.5*\va) {$p_{12*} $};

				\node at (3.45*\ha, .75*\va) {$p'_{1*}p_{2*} $};
				\node at (1.7*\ha, .85*\va) {$i^!_{12}j^*_2 $};
				\node at (2.95*\ha, 1.85*\va) {$j^*_1 $};
				\node at (2.0*\ha, 1.85*\va) {$d'^*_G $};
				\node at (.55*\ha, 1.7*\va) {$k^! $};
				\node at (3.4*\ha, 2.70*\va) {$d^*_G $};
				\node at (1.65*\ha, 2.7*\va) {$\ell^* $};
				\node at (3.05*\ha, 3.65*\va) {$q_* $};
				\node at (2.1*\ha, 3.5*\va) {$m_{G*} $};
				\node at (.9*\ha, 2.7*\va) {$m'_{G*} $};

				\foreach \ma/\na/\mb/\nb/\curvefactor in {4/1/1/4/0.7, 1/1/4/4/0.7, 3/0/0/3/0.7, 0/2/3/5/0.7, 2/0/4/2/0.7, 4/3/2/5/0.7} {
					\draw ({\ma*\ha}, {\na*\va}) to[out=270,in=90,looseness=\curvefactor] ({\mb*\ha}, {\nb*\va});
					
				}
				
				\foreach \m/\na/\nb in {4/0/1, 4/2/3, 4/4/5, 1/0/1, 1/4/5, 0/0/2, 0/3/5} {
					\draw ({\m*\ha}, {\na*\va}) to ({\m*\ha}, {\nb*\va});
				}\\
			};
		\end{tikzpicture}
	\end{equation}
	These are respectively the dual graphs of the backward- and forward-facing exterior faces of the following composite of (\ref{eq:equivgrass}), (\ref{eq:pconvsquare}), and (\ref{eq:ijconvsquare}). 
	\begin{equation}\label{eq:fourcubes}
		\begin{tikzpicture}[baseline=(current  bounding  box.center),thick,>=\arrtip]
			\newcommand*{\ha}{2.0}; \newcommand*{\hb}{2.2}; 
			\newcommand*{\hc}{2.0}; \newcommand*{\hd}{2.2}; 
			\newcommand*{\he}{2.0};
			\newcommand*{\va}{-1.2}; \newcommand*{\vb}{-1.4}; 
			\newcommand*{\vc}{-1.2}; \newcommand*{\vd}{-1.4}; 
			\newcommand*{\ve}{-1.2}; 
			
			\node (ab) at (\ha,0) {${}_{\wL}\Gr_{\wL} \tighttimes {}_{\wL}\Gr_{\wL}$};
			\node (ad) at (\ha+\hb+\hc,0) {${}_{\wL}\Gr_{\wG} \tighttimes {}_{\wG}\Gr_{\wG}$};
			\node (af) at (\ha+\hb+\hc+\hd+\he,0) {${}_{\wG}\Gr_{\wG} \tighttimes {}_{\wG}\Gr_{\wG}$};
			\node (ba) at (0,\va) {${}_{L}\Gr_{L} \tighttimes {}_{L}\Gr_{L}$};
			\node (bc) at (\ha+\hb,\va) {${}_{L}\Gr_{G} \tighttimes {}_{G}\Gr_{G}$};
			\node (be) at (\ha+\hb+\hc+\hd,\va) {${}_{G}\Gr_{G} \tighttimes {}_{G}\Gr_{G}$};
			
			\node (cb) at (\ha,\va+\vb) {${}_{\wL}\Gr_{\wL}\Gr_{\wL}$};
			\node (cd) at (\ha+\hb+\hc,\va+\vb) {${}_{\wL}\Gr_{\wG}\Gr_{\wG}$};
			\node (cf) at (\ha+\hb+\hc+\hd+\he,\va+\vb) {${}_{\wG}\Gr_{\wG}\Gr_{\wG}$};
			\node (da) at (0,\va+\vb+\vc) {${}_{L}\Gr_{L}\Gr_{L}$};
			\node (dc) at (\ha+\hb,\va+\vb+\vc) {${}_{L}\Gr_{G}\Gr_{G}$};
			\node (de) at (\ha+\hb+\hc+\hd,\va+\vb+\vc) {${}_{G}\Gr_{G}\Gr_{G}$};
			
			\node (eb) at (\ha,\va+\vb+\vc+\vd) {${}_{\wL}\Gr_{\wL}$};
			\node (ed) at (\ha+\hb+\hc,\va+\vb+\vc+\vd) {${}_{\wL}\Gr_{\wG}$};
			\node (ef) at (\ha+\hb+\hc+\hd+\he,\va+\vb+\vc+\vd) {${}_{\wG}\Gr_{\wG}$};
			\node (fa) at (0,\va+\vb+\vc+\vd+\ve) {${}_{L}\Gr_{L}$};
			\node (fc) at (\ha+\hb,\va+\vb+\vc+\vd+\ve) {${}_{L}\Gr_{G}$};
			\node (fe) at (\ha+\hb+\hc+\hd,\va+\vb+\vc+\vd+\ve) {${}_{G}\Gr_{G}$};
			
			\draw[->] (ab) to node[above,pos=.5] {$i' \tighttimes j'i' $} (ad);
			\draw[->] (ad) to node[above,pos=.5] {$j' \tighttimes \id $} (af);
			\draw[->] (cb) to node[above,pos=.3] {$k' $} (cd);
			\draw[->] (cd) to node[above,pos=.3] {$\ell' $} (cf);
			\draw[->] (eb) to node[above,pos=.3] {$i'$} (ed);
			\draw[->] (ed) to node[above,pos=.3] {$j' $} (ef);
			
			\draw[<-] (ab) to node[right,pos=.7] {$ $} (cb);
			\draw[<-] (ad) to node[right,pos=.7] {$ $} (cd);
			\draw[<-] (af) to node[right,pos=.5] {$ $} (cf);
			\draw[->] (cb) to node[right,pos=.7] {$ $} (eb);
			\draw[->] (cd) to node[right,pos=.7] {$ $} (ed);
			\draw[->] (cf) to node[right,pos=.5] {$ $} (ef);
			
			\draw[->] (ab) to node[above left, pos=.7] {$p'' \tighttimes p'' $} (ba);
			\draw[->] (ad) to node[above left, pos=.85] {$p' \tighttimes p $} (bc);
			\draw[->] (af) to node[above left, pos=.85] {$p \tighttimes p$} (be);
			\draw[->] (cb) to node[above left, pos=.6] {$q'' $} (da);
			\draw[->] (cd) to node[above left, pos=.75] {$q' $} (dc);
			\draw[->] (cf) to node[above left, pos=.75] {$q $} (de);
			\draw[->] (eb) to node[above left, pos=.6] {$p'' $} (fa);
			\draw[->] (ed) to node[above left, pos=.75] {$p' $} (fc);
			\draw[->] (ef) to node[above left, pos=.75] {$p $} (fe);
			
			\draw[-,line width=6pt,draw=white] (ba) to  (bc);
			\draw[->] (ba) to node[above,pos=.3] {$ $} (bc); 
			\draw[-,line width=6pt,draw=white] (bc) to  (be);
			\draw[->] (bc) to node[above,pos=.3] {$ $} (be);
			\draw[-,line width=6pt,draw=white] (da) to  (dc);
			\draw[->] (da) to node[above,pos=.3] {$ $} (dc);
			\draw[-,line width=6pt,draw=white] (dc) to  (de);
			\draw[->] (dc) to node[above,pos=.3] {$ $} (de);
			\draw[-,line width=6pt,draw=white] (fa) to  (fc);
			\draw[->] (fa) to node[above,pos=.5] {$ $} (fc);
			\draw[-,line width=6pt,draw=white] (fc) to  (fe);
			\draw[->] (fc) to node[above,pos=.5] {$ $} (fe);
			
			\draw[-,line width=6pt,draw=white] (ba) to  (da);
			\draw[<-] (ba) to node[left] {$ $} (da);
			\draw[-,line width=6pt,draw=white] (da) to  (fa);
			\draw[->] (da) to node[left] {$ $} (fa);
			\draw[-,line width=6pt,draw=white] (bc) to  (dc);
			\draw[<-] (bc) to node[left] {$ $} (dc);
			\draw[-,line width=6pt,draw=white] (dc) to  (fc);
			\draw[->] (dc) to node[left] {$ $} (fc);
			\draw[-,line width=6pt,draw=white] (be) to  (de);
			\draw[<-] (be) to node[left] {$ $} (de);
			\draw[-,line width=6pt,draw=white] (de) to  (fe);
			\draw[->] (de) to node[left] {$ $} (fe);
		\end{tikzpicture}
	\end{equation}	
	The diagrams (\ref{eq:stringdiags}) differ by a sequence of four Yang-Baxter moves, corresponding to the four cubes in (\ref{eq:fourcubes}). It thus suffices to verify the local identities corresponding to each of these moves. But each of these follows from the compatibility of adjunctions and their associated base change maps with composition. For example, the leftmost move (corresponding to the bottom-left cube) follows since the identities $m_{L*}q''_* \cong p''_* m_{\wL *}$ and $m'_{G*}q'_* \cong p'_* m'_{\wG *}$ imply an identity of base change transformations:
		\begin{equation*}
		\begin{tikzpicture}
			[baseline=(current  bounding  box.center),thick,>=\arrtip]
			\newcommand*{\ha}{3}; \newcommand*{\hb}{3};
			\newcommand*{\va}{-1.5};
			\node (aa) at (0,0) {$m_{L*} q''_* k'^!$};
			\node (ab) at (\ha,0) {$m_{L*} k^! q'_*$};
			\node (ac) at (\ha+\hb,0) {$i^! m'_{G*} q'_*$};
			\node (ba) at (0,\va) {$p''_* m_{\wL*} k'^!$};
			\node (bb) at (\ha,\va) {$p''_* i'^! m'_{\wG*}$};
			\node (bc) at (\ha+\hb,\va) {$i^! p'_* m_{\wG*}.$};
			\draw[->] (aa) to node[above] {$\sim $} (ab);
			\draw[->] (ab) to node[above] {$ $} (ac);
			\draw[->] (ba) to node[above] {$ $} (bb);
			\draw[->] (bb) to node[above] {$\sim $} (bc);
			\draw[->] (aa) to node[below,rotate=90] {$\sim $} (ba);
			\draw[->] (ac) to node[below,rotate=90] {$\sim $} (bc);
		\end{tikzpicture} 
	\end{equation*}
	Only the top move (or top-left cube) is more involved, in that it follows by applying this principle three times, once for each of the smooth base change isomorphisms appearing in~(\ref{eq:threesbc}). 
\end{proof}

We close the section with some informal technical remarks. The category $D({}_{G}\Gr_{G})$ can be realized alternatively in terms of *-sheaves or !-sheaves, each highlighting complementary correspondence functorialities. In terms of correspondence categories we have functors
$$D^!: \Corr(\indGStk)_{all,prop} \to \PrL, \quad D_*: \Corr(\indGStk)_{sm,all} \to \PrL,$$
encoding that $D^!(-)$ admits !-pullback along arbitrary maps and pushforward along ind-proper maps (and that these satisfy base change), while $D_*(-)$ admits *-pullback along smooth maps and pushforward along arbitrary maps (and that these satisfy base change). These correspondence functorialities are elementary, in the sense that the relevant base change maps arise from adjunctions. 

Either correspondence functoriality can be used to describe convolution in $D({}_{G}\Gr_{G})$, since this only involves smooth pullback and ind-proper pushforward. On the other hand, the fact that convolution is compatible with $p_{*}$ is only manifest in terms of *-sheaves, in the sense that $p: {}_{\wG}\Gr_{\wG} \to {}_{G}\Gr_{G}$ is not ind-proper, hence it should induce a map of algebra objects in $\Corr(\indGStk)_{sm,all}$ but not $\Corr(\indGStk)_{all,prop}$ (more precisely, a lax map of algebra objects in suitable 2-categorical enhancements). Meanwhile, $(ji)^{\op}: {}_{G}\Gr_{G} \to {}_{L}\Gr_{L}$ induces a map of algebra objects in $\Corr(\indGStk)_{all,prop}$ but not $\Corr(\indGStk)_{sm,all}$, so the fact that convolution is compatible with $(ji)^!_{ren} = i^!j^*$ is only manifest in terms of !-sheaves (and is comparatively obscure in our use of *-sheaves above). 

It would be convenient to combine these two correspondence functors. 
That is, let $\indGStk^\omega$ denote the category of ind-geometric stacks $X$ equipped with a renormalized dualizing sheaf $\omega_X \in D_*(X)$, in the sense of an object such that the induced functor $D^!(X) \to D_*(X)$, $\cF \mapsto \cF \overset{!}{\otimes} \omega_X$ is an equivalence. Such an $\omega_X$ exists whenever $X$ is placid. We would like a functor
$ D: \Corr(\indGStk^\omega)_{fwd,bkwd} \to \PrL$
which takes $X \xleftarrow{h} Y \xrightarrow{f} Z$ to 
$$D_*(X) \cong D^!(X) \xrightarrow{h^!} D^!(Y) \cong D_*(Y) \xrightarrow{f_*} D_*(Z).$$
Here $fwd$ and $bkwd$ are some classes of maps containing $p$ and $ji$, and the correspondences considered should satisfy some compatibility with the renormalized dualizing sheaves. 

Such a functor would encode a package of coherence data expressing a compatibility between pushforward and renormalized !-pullback along the given classes of maps. The proof above essentially constructs a small piece of such a package in an ad hoc way (and suggests, for example, that we could take $fwd$ to be all maps and $bkwd$ to be maps factoring as an ind-proper map followed by a smooth map). 
In particular, such a functor would conveniently encapsulate the more awkward parts of the proof, which result from the explicit factorization of renormalized !-pullbacks into left and right adjoints of pushforwards. Moreover, a suitable 2-categorical enhancement would in principle allow for an immediate formal proof of the claim. 

\section{Proofs of main results}
We now turn to the proofs of Theorems \ref{thm:mainthmintro}, \ref{thm:Weylaction}, and \ref{thm:levithmintro}. For the rest of the paper we fix $n \in \N$, let $G = PGL_n$, let $T \subset G$ be its standard Cartan subgroup, and fix a standard Levi subgroup $L \subset G$. We set
\begin{equation}\label{eq:N}
N = \Hom(\C,\C^2) \oplus \Hom(\C^2,\C^3) \oplus \cdots \oplus \Hom(\C^{n-1}, \C^n),
\end{equation}
which is naturally a representation of $\Pi_{k=1}^n GL_k$. The diagonal subgroup $\G_m \subset \G_m^n \cong \Pi_{k=1}^n Z(GL_k)$ acts trivially, hence $N$ inherits an action of 
\begin{equation}\label{eq:wG}
	\wG := (GL_1 \times \GL_2 \times ... \times \GL_n)/\G_m.
\end{equation}
The projection $\Pi_{k=1}^n GL_k \to GL_n \to G$ factors through a projection $\wG \to G$.

We define a subgroup $\wL \subset \wG$ by the following Cartesian square of algebraic groups.
\begin{equation}\label{eq:Tgroupsquare}
	\begin{tikzpicture}[baseline=(current  bounding  box.center),thick,>=\arrtip]
		\newcommand*{\ha}{3}; \newcommand*{\hb}{3};
		\newcommand*{\va}{-1.5}; \newcommand*{\vb}{-1.5}; 
		
		\node (aa) at (0,0) {$\wL$};
		\node (ab) at (\ha,0) {$\wG$};
		
		\node (ba) at (0,\va) {$L$};
		\node (bb) at (\ha,\va) {$G$};
		
		\draw[->] (aa) to node[above,pos=.5] {$ $} (ab);
		\draw[->] (ba) to node[above,pos=.5] {$ $} (bb);
		
		\draw[->] (aa) to node[right,pos=.5] {$ $} (ba);
		\draw[->] (ab) to node[right,pos=.5] {$ $} (bb);
	\end{tikzpicture}
\end{equation}
As in the previous section, we notate the associated diagram of loop-equivariant Grassmannians as
\begin{equation}\label{eq:equivgrassT}
	\begin{tikzpicture}[baseline=(current  bounding  box.center),thick,>=\arrtip]
		\newcommand*{\ha}{3}; \newcommand*{\hb}{3};
		\newcommand*{\va}{-1.5}; \newcommand*{\vb}{-1.5}; 
		
		\node (aa) at (0,0) {${}_{\wL}\Gr_{\wL}$};
		\node (ab) at (\ha,0) {${}_{\wL}\Gr_{\wG}$};
		\node (ac) at (\ha+\hb,0) {${}_{\wG}\Gr_{\wG}$};
		
		\node (ba) at (0,\va) {${}_{L}\Gr_{L}$};
		\node (bb) at (\ha,\va) {${}_{L}\Gr_{G}$};
		\node (bc) at (\ha+\hb,\va) {${}_{G}\Gr_{G},$};
		
		\draw[->] (aa) to node[above,pos=.5] {$i' $} (ab);
		\draw[->] (ab) to node[above,pos=.5] {$j' $} (ac);
		\draw[->] (ba) to node[above,pos=.5] {$i $} (bb);
		\draw[->] (bb) to node[above,pos=.5] {$j $} (bc);
		
		\draw[->] (aa) to node[right,pos=.5] {$p'' $} (ba);
		\draw[->] (ab) to node[right,pos=.5] {$p' $} (bb);
		\draw[->] (ac) to node[right,pos=.5] {$p $} (bc);
	\end{tikzpicture}
\end{equation}
where e.g. ${}_L \Gr_G := \hL_\cO \backslash \hG_\cK / \hG_\cO$. 

Recall that $L$, being a standard Levi subgroup, corresponds to an ordered tuple $\vec{n} = (n_1,\dotsc,n_k)$ summing to $n$ for which $L$ has block diagonal decomposition
$$ L \cong (GL_{n_1} \times \cdots \times GL_{n_k})/\G_m.$$
Explicitly, if $\fl$ is the Lie algebra of $L$, let $e_{i_1}, \dotsc, e_{i_{k-1}} \subset \fu := \mathrm{Lie}(U)$ be the generators of $\fu$ not belonging to $\fl$, written with increasing indices. Then  $n_j = i_j - i_{j-1}$, where $i_0 = 0$ and $i_k = n$. 
In terms of quiver gauge theories, the pair $(\wL, N)$ is then associated to the quiver
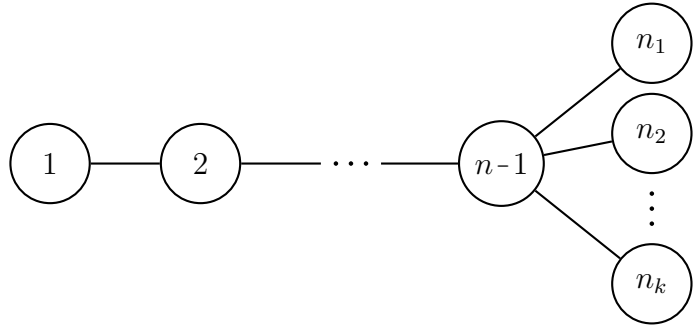
\begin{equation*}
	\begin{tikzpicture}[baseline=(current bounding box.center), thick, >=Stealth]
		\node[circle, draw, minimum size=30] (1) at (0,0) {$1$};
		\node[circle, draw, minimum size=30] (2) at (2,0) {$2$};
		\fill (3.8,0) circle (.03);
		\fill (4,0) circle (.03);
		\fill (4.2,0) circle (.03);
		\fill (8,-0.4) circle (.03);
		\fill (8,-0.6) circle (.03);
		\fill (8,-0.8) circle (.03);
		\node[circle, draw, minimum size=30] (n-1) at (6,0) {$n\! -\! 1$};
		\node [circle, draw, minimum size=30] (v1) at (8,1.6) {$n_1$};
		\node [circle, draw, minimum size=30] (v2) at (8,0.4) {$n_2$};
		\node [circle, draw, minimum size=30] (v4) at (8,-1.6) {$n_k$};
		
		\draw (2) -- (1);
		\draw (2) -- (3.6,0);
		\draw (n-1) -- (4.4,0);
		\draw (n-1) -- (v1);
		\draw (n-1) -- (v2);
		\draw (n-1) -- (v4);
	\end{tikzpicture}
\end{equation*}
In particular, when $L = T$ the resulting pair $(\wT,N)$ corresponds to the quiver
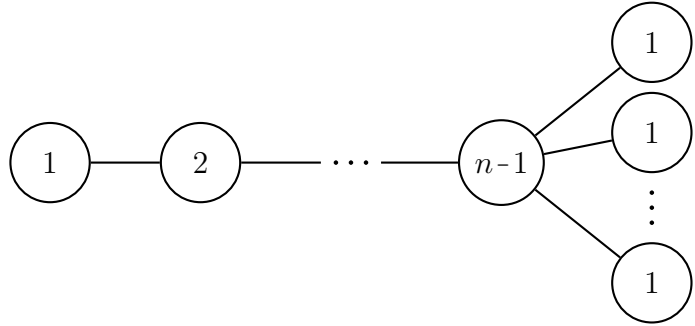
\begin{equation*}
	\begin{tikzpicture}[baseline=(current bounding box.center), thick, >=Stealth]
		\node[circle, draw, minimum size=30] (1) at (0,0) {$1$};
		\node[circle, draw, minimum size=30] (2) at (2,0) {$2$};
		\fill (3.8,0) circle (.03);
		\fill (4,0) circle (.03);
		\fill (4.2,0) circle (.03);
		\fill (8,-0.4) circle (.03);
		\fill (8,-0.6) circle (.03);
		\fill (8,-0.8) circle (.03);
		\node[circle, draw, minimum size=30] (n-1) at (6,0) {$n\! -\! 1$};
		\node [circle, draw, minimum size=30] (v1) at (8,1.6) {$1$};
		\node [circle, draw, minimum size=30] (v2) at (8,0.4) {$1$};
		\node [circle, draw, minimum size=30] (v4) at (8,-1.6) {$1$};
		
		\draw (2) -- (1);
		\draw (2) -- (3.6,0);
		\draw (n-1) -- (4.4,0);
		\draw (n-1) -- (v1);
		\draw (n-1) -- (v2);
		\draw (n-1) -- (v4);
	\end{tikzpicture}
\end{equation*}
and when $L = G$ the resulting pair $(\wG,N)$ corresponds to the quiver 
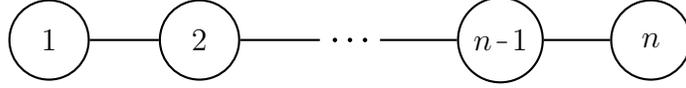
\begin{equation*}
	\begin{tikzpicture}[baseline=(current bounding box.center), thick, >=Stealth]
		\node[circle, draw, minimum size=30] (1) at (0,0) {$1$};
		\node[circle, draw, minimum size=30] (2) at (2,0) {$2$};
		\fill (3.8,0) circle (.03);
		\fill (4,0) circle (.03);
		\fill (4.2,0) circle (.03);
		\node[circle, draw, minimum size=30] (n-1) at (6,0) {$n\! -\! 1$};
		\node [circle, draw, minimum size=30] (v1) at (8,0) {$n$};
		
		\draw (2) -- (1);
		\draw (2) -- (3.6,0);
		\draw (n-1) -- (4.4,0);
		\draw (n-1) -- (v1);
	\end{tikzpicture}
\end{equation*}

Associated to these quiver gauge theories we have ring objects $\cA_{\wG,N} := \pi_*(\omega_{\wG,N})\in D({}_{\wG} \Gr_{\wG})$ and $\cA_{\wL,N} := \pi_*(\omega_{\wL,N})\in D({}_{\wL} \Gr_{\wL})$. These are related by the following statement (see also \cite[Sec. 5(xi)]{BFN19}) which, for example, follows from \cite[Lem. 4.2.1]{BDFRT22}, which states that both are obtained by !-restriction from the universal ring object on $\Gr_{\Sp_{N \oplus N^*}}$, hence they are related by !-restriction to each other. 

\begin{Proposition}\label{prop:!restrict}
There is an isomorphism $\cA_{\wL,N} \cong i'^! j'^*(\cA_{\wG,N})$ of ring objects in $D({}_{\wL} \Gr_{\wL})$. 
\end{Proposition}

Next we recall the following result of Braverman-Finkelberg-Nakajima, which relates $\cA_{\wG,N}$ to the ring object $\cA_{reg} \in D({}_G \Gr_G)$ corresponding to the regular representation of $SL_n$ under geometric Satake. 

\begin{Theorem}\label{thm:BFNresult} \cite[Thm. 2.11]{BFN19} There is an isomorphism $p_*(\cA_{\wG,N}) \cong \cA_{reg}$ of ring objects in $D({}_G \Gr_G)$. 
\end{Theorem}

More precisely, \cite[Thm. 2.11]{BFN19} as stated does not include loop equivariance, but the loop-equivariant version follows from it: both sides are in the heart  of the ind-perverse t-structure, and the forgetful functor from $G_\cO \rtimes \C^\times$-equivariant ind-perverse sheaves to $G_\cO$-equivariant ind-perverse sheaves is fully faithful \cite[Prop. A.1]{MV07}. 

Finally, we recall the following result of Ginzburg-Riche, which specifically is obtained by substituting $\cA_{reg}$ into the theorem of \cite[Section 1.7]{GR15} (and unwinding that the ring structure implicit there is given by the procedure described in the previous section). Here $D_\hbar(SL_n/U)$ is the ring of asymptotic differential operators on $SL_n/U$, i.e. the Rees algebra of the ordinary ring of differential operators with respect to its order filtration. We use $H^\bullet$ to denote total cohomology, i.e. pushing forward to a point and then summing cohomology groups. Note that this implicitly refers to equivariant cohomology when, as in the following statement, it is applied to sheaves on stacks. 

\begin{Theorem}\label{thm:GRresult} 
	When $L = T$ there is a ring isomorphism $H^\bullet i^!j^*(\cA_{reg}) \cong D_\hbar(SL_n/U)$. 
\end{Theorem}

We now obtain Theorem \ref{thm:mainthmintro} by assembling these results with the result of the previous section. 

\begin{proof}[Proof of Theorem \ref{thm:mainthmintro}]
We have ring isomorphisms
\begin{align*}
\C_\hbar[\cM(Q_n)] &\cong H^\bullet(\cA_{\wT,N})\\ 
&\cong H^\bullet i'^!j'^*(\cA_{\wG,N})\\ 
&\cong H^\bullet  i^! j^* p_*(\cA_{\wG,N})\\ 
&\cong H^\bullet  i^! j^*(\cA_{reg})\\ 
&\cong D_\hbar(SL_n/U).
\end{align*}
The first is a definition, while the remaining ones follow respectively from Proposition \ref{prop:!restrict}, Proposition \ref{prop:algiso}, Theorem \ref{thm:BFNresult}, and Theorem \ref{thm:GRresult}. 
\end{proof}

We prove Theorem \ref{thm:Weylaction} together with a few other compatibilities of the isomorphism $\C_\hbar[\cM(Q_n)] \cong D_\hbar(SL_n/U)$, packaged together into Theorem \ref{thm:compatibilities} below. These compatibilities involve the action of $T^\vee \subset SL_n$ on $D_\hbar(SL_n/U)$ induced by its right action on $SL_n/U$. The data of this action is equivalent to that of the induced grading by~$P^\vee$, the weight lattice of $T^\vee$ and coweight lattice of $T$. The quantum moment map of this action also gives $D_\hbar(SL_n/U)$ the structure of a $\C_\hbar[\ft]$-algebra (where $\C_\hbar[\ft] \cong \C[\ft] \otimes \C[\hbar]$). It follows from \cite[Lem. 3.2.1]{GR15} that these structures intertwine the Gelfand-Graev $S_n$-action on $D_\hbar(SL_n/U)$ with the Weyl group $S_n$-actions on $P^\vee$ and $\C_\hbar[\ft]$. 

On the other hand, $\C_\hbar[\cM(Q_n)]$ has a grading by $\pi_0(\Gr_{\wT}) \cong \pi_1(\wT)$, hence by its quotient $P^\vee \cong \pi_1(T)$, and is an algebra over $\C_\hbar[\ft] \cong H^\bullet_{T \times \C^\times}(\pt)$. In other words, pushforward from $D({}_{\wT} \Gr_{\wT})$ to a point factors through $D({}_T \Gr_T)$, but $D({}_T \Gr_T) \to D(\pt)$ factors through $P^\vee$-graded $\C_\hbar[\ft]$-modules since ${}_T \Gr_T \cong P^\vee \times B\hT_\cO$. To make explicit the compatibility with the $S_n$-action, write $N_T \cong T \rtimes S_n \subset G$ for the normalizer of $T$ and $\wN_T$ for its preimage in $\wG$. Then likewise $D({}_{N_T} \Gr_{N_T}) \cong D((P^\vee \times B\hT_\cO)/S_n) \to D(BS_n)$ factors through $S_n$-equivariant $P^\vee$-graded $\C_\hbar[\ft]$-modules. But $\cA_{\wT,N}$ is obtained from $\cA_{\wN_T,N}$ by forgetting equivariance (note ${}_{\wN_T} \Gr_{\wN_T} \cong {}_{\wN_T} \Gr_{\wT}$), hence by base change its cohomology is obtained from the pushforward of $\cA_{\wN_T,N}$ to $BS_n$ by forgetting equivariance. 

\begin{Theorem}\label{thm:compatibilities}
	The isomorphism $\C_\hbar[\cM(Q_n)] \cong D_\hbar(SL_n/U)$ of Theorem \ref{thm:mainthmintro} is an isomorphism of $S_n$-equivariant $P^\vee$-graded $\C_\hbar[\ft]$-algebras. 
\end{Theorem}

\begin{proof}[Proof of Theorem \ref{thm:Weylaction}]
Repeating the proof of Theorem \ref{thm:mainthmintro} with $N_T$ and $\wN_T$ in place of $T$ and $\wT$ we obtain an isomorphism $p'_*(\cA_{\wN_T,N}) \cong i^! j^*(\cA_{reg})$ of ring objects in $D({}_{N_T} \Gr_{N_T})$. Following the remarks before the statement, the isomorphism between the cohomologies of these ring objects immediately enhances to an isomorphism of $S_n$-equivariant $P^\vee$-graded $\C_\hbar[\ft]$-algebras.  That the resulting structures on $D_\hbar(SL_n/U)$ are induced by the Gelfand-Graev action and the right $T^\vee$-action on $SL_n/U$ follows from \cite[Thm. 2.5.5, Lem. 3.2.1]{GR15}. 
\end{proof}

Next we recall the following generalization by Macerato of the semiclassical limit of Theorem~\ref{thm:GRresult}. Here we let $\psi \in \fu^*$ denote the unique character such that $\psi(e_i) = 0$ if $i = n_1 + \cdots + n_j$ for some $j$, and $\psi(e_i) = 1$ otherwise. The action of $U$ on $SL_n$ induces a moment map $T^* SL_n \to \fu^*$, and associated to $\psi$ we have the Hamiltonian reduction 
$$ T^* SL_n \sslash_\psi U := (T^*SL_n \times_{\fu^*} {\psi}) / U. $$

\begin{Theorem}\label{thm:Macresult}
	\cite[Thm. 1.5.2]{Mac23} There is a ring isomorphism $H^\bullet i^!j^*(\cA_{reg})|_{\hbar = 0} \cong \C[T^*SL_n \sslash_{\psi} U]$ .
\end{Theorem}

\begin{proof}[Proof of Theorem \ref{thm:levithmintro}]
	The same as the proof of Theorem \ref{thm:mainthmintro}, except we use Theorem \ref{thm:Macresult} in place of Theorem \ref{thm:GRresult}. 
\end{proof}

We conclude with remarks on potential extensions of our main results.

\begin{Remark}
The variety $T^*SL_n \sslash_{\psi} U$ admits an action of $\cM(L, 0)$, see \cite[Sec. 3.2]{Mac23}. Here, we use the identification of $\cM(L, 0)$ with the group scheme of regular centralizers for~$L^{\vee}$ \cite{BFM05}. On the other hand, the Coulomb branch $\cM_C(Q_{\vec{n}})$ admits an action of $\cM(\wL, 0)$, see e.g. \cite[Thm. 1(iv)]{Tel22}. We intend to compare these constructions in future work. 
\end{Remark}

\begin{Remark}\label{Bonus Expectations of Theorem} We expect the following additional compatibilities of Theorem \ref{thm:mainthmintro}. 
\begin{enumerate}
\item Recall that, as explained in \cite[Section 8]{DHK21}, the \lq leftmost\rq{} $n - 1$ nodes of the quiver $Q_{n}$ are \textit{balanced} in the sense of e.g. \cite[Section A(ii)]{BFN19}. 
Therefore, the construction of \cite[Section A(ii)]{BFN19} implies that there is a natural Lie algebra embedding $\mathfrak{sl}_n \subseteq \C_\hbar[\cM(Q_n)]$ for all $n \geq 1$. We expect that, under the isomorphism of \cref{thm:mainthmintro}, this subalgebra coincides with the Lie subalgebra $\mathfrak{sl}_n \subseteq D_\hbar(SL_n/U)$ of vector fields induced from the $SL_n$-action on $SL_n/U$. 
\item Recall that the \textit{Gelfand-Tsetlin subalgebra} of $U_{\hbar}(\mathfrak{gl}_n)$ is the commutative subalgebra $R_n$ of $U_{\hbar}(\mathfrak{gl}_n)$ generated by $Z(U_{\hbar}(\mathfrak{gl}_p))$ for $p \in \{1, ..., n\}$ where we view $\mathfrak{gl}_p \subseteq \mathfrak{gl}_n$ by the embedding sending a $p \times p$ matrix $M$ to the $n \times n$ matrix whose $(i, j)$-entry is that of $M$ if both $i \leq p$ and $j \leq p$ and is zero otherwise. Let $R_n' := R_n \cap U_{\hbar}(\mathfrak{sl}_n)$. We expect that the canonical quantum integrable system of \labelcref{ComAlg} corresponds under Theorem \ref{thm:mainthmintro} to the composite \[R_n' \otimes_{Z(U_{\hbar}(\mathfrak{sl}_n))} U_{\hbar}(\mathfrak{t}) \subseteq U_{\hbar}(\mathfrak{sl}_n) \otimes_{Z(U_{\hbar}(\mathfrak{sl}_n))} U_{\hbar}(\mathfrak{t}) \to D_\hbar(SL_n/U)\] of the inclusion of $R_n'$ and the comoment map for the action of $SL_n \times T$ on $SL_n/U$. 
\end{enumerate}
\end{Remark}

\begin{Remark}
We let $B$ denote the subgroup of $\SL_n$ of upper triangular matrices of determinant 1, and let $\mathfrak{b}$ denote its Lie algebra, which may be equivalently described as the set of traceless upper triangular matrices. We define a left action of $U$ on $B$ by the formula $u \cdot b := ubu^{-1}$ and a right action of $U$ on $\SL_n$ through right multiplication, or in other words by the formula $g \cdot u := gu$. With these actions, we may form the balanced product $X_n := \SL_n \times^U B$, the variety informally obtained from $\SL_n \times B$ by identifying $(gu, b) \sim (g, ubu^{-1})$. The variety $X_n$ is quasi-affine (for example, by essentially identical arguments to that of \cite[Corollary 5.6]{GannonAProofofGinzburgKazhdanConjecture}) and so the scheme 
 \[\overline{X_n} := \mathrm{Spec}(\C[\SL_n \times^U B]) \cong \mathrm{Spec}(\C[\SL_n \times B]^U)\] contains $X_n$ as an open subscheme. We conjecture that the $K$-theoretic Coulomb branch of the quiver $Q_n$ with gauge group $\tilde{T}$ is isomorphic to $\overline{X_n}$.\footnote{This conjecture has also been independently obtained by Ana B\u alibanu.}

As some motivation for this conjecture we recall that there is an $\SL_n$-equivariant isomorphism \begin{equation}\label{Cotangent Bundle is Quotient of bBundle}T^*(\SL_n/U) \cong \SL_n \times^U \mathfrak{b}\end{equation} induced by a choice of nondegenerate $\SL_n$-invariant bilinear form $\kappa$ on $\mathfrak{sl}_n$, such as the Killing form. Indeed, there is a (well known) $\SL_n$-equivariant isomorphism \begin{equation}\label{Iso One}\SL_n \times^U (\mathfrak{g}/\mathfrak{u})^* \xrightarrow{\sim} T^*(\SL_n/U)\end{equation} induced from the $\SL_n$-equivariance of the tangent sheaf. Moreover, since $\kappa$ induces an isomorphism $\mathfrak{sl}_n \xrightarrow{\sim} \mathfrak{sl}_n^*$, it is not difficult to prove (see for example \cite[Lemma 2.3.4]{RicheKostantSection}) that there is an induced $B$-equivariant isomorphism $\mathfrak{b} \xrightarrow{\sim} (\mathfrak{sl}_n/\mathfrak{u})^*$. We therefore obtain an isomorphism \begin{equation}\label{Iso Two}\SL_n \times^U \mathfrak{b} \xrightarrow{\sim} \SL_n \times^U (\mathfrak{sl}_n/\mathfrak{u})^* \end{equation} which is manifestly $\SL_n$-equivariant. Combining \labelcref{Iso One} and \labelcref{Iso Two} we recover the isomorphism \labelcref{Cotangent Bundle is Quotient of bBundle}. Therefore, our conjecture on $\overline{X_n}$ can be regarded as a \lq multiplicative\rq{} variant of the semiclassical version of \cref{thm:mainthmintro}.

If $B' \subseteq \GL_n$ denotes the subgroup of all upper triangular matrices, we moreover conjecture that the affine closure \[\overline{Y_n} := \mathrm{Spec}(\C[\GL_n \times^U B']) \cong \mathrm{Spec}(\C[\GL_n \times B']^U)\] of the quasi-affine variety $Y_n := \GL_n \times^U B'$ is isomorphic to the $K$-theoretic Coulomb branch of the quiver $Q_n$ and gauge group $\GL_1^n \times \prod_{i = 1}^{n - 1} \GL_i .$
\end{Remark}

\begin{Remark}\label{rem:BFNnilconecomparison}
	Suppose we replace $L$ with the trivial subgroup $1 \subset G$ in (\ref{eq:Tgroupsquare}). Then the subgroup $\wL$ becomes the kernel $\wone := \ker(\wG \to G)$, the map $i$ becomes the inclusion of the identity point into $\Gr_G$, and the map $j$ becomes the quotient $\Gr_G \to {}_G \Gr_G$. The pair $(\wone,N)$ corresponds to the framed quiver 
	\begin{equation*}
		\begin{tikzpicture}[baseline=(current bounding box.center), thick, >=Stealth]
			\node[circle, draw, minimum size=30] (1) at (0,0) {$1$};
			\node[circle, draw, minimum size=30] (2) at (2,0) {$2$};
			\fill (3.8,0) circle (.03);
			\fill (4,0) circle (.03);
			\fill (4.2,0) circle (.03);
			\node[circle, draw, minimum size=30] (n-1) at (6,0) {$n\! -\! 1$};
			\node [rectangle, draw, minimum size=30] (v1) at (8,0) {$n$};
			
			\draw (2) -- (1);
			\draw (2) -- (3.6,0);
			\draw (n-1) -- (4.4,0);
			\draw (n-1) -- (v1);
		\end{tikzpicture}
	\end{equation*}
	 By \cite[Thm. 7.3.1]{ABG04}, the ring $i^! j^*(\cA_{reg})$ is now isomorphic to the coordinate ring $\C[\cN]$ of the nilpotent cone of $SL_n$. Using this in place of the results of \cite{GR15} and \cite{Mac23}, the proofs above now specialize to the proof implicit in \cite[Sec. 2(iv)]{BFN19} that there is an isomorphism $\cM_{\wone,N} \cong \cN$ (such an isomorphism was already established by different means in \cite{BFN19b}). In particular, Propositions \ref{prop:algiso} and \ref{prop:!restrict} function in the proofs above as a generalization of \cite[Lem. 2.9]{BFN19} to the case where the inclusion of the trivial subgroup is replaced by a general homomorphism $L \to G$. The interpretation of the Springer resolution in \cite[Sec. 2(vi)]{BFN19} is yet another variation on this theme, with a suitable subset $\bN \subset \Gr_T$ in the role played above by ${}_T \Gr_T$ and the identity point. 
\end{Remark}

\bibliographystyle{amsalpha}
\bibliography{bibfile}

\newcommand{\etalchar}[1]{$^{#1}$}
\providecommand{\bysame}{\leavevmode\hbox to3em{\hrulefill}\thinspace}
\providecommand{\MR}{\relax\ifhmode\unskip\space\fi MR }
\providecommand{\MRhref}[2]{%
  \href{http://www.ams.org/mathscinet-getitem?mr=#1}{#2}
}
\providecommand{\href}[2]{#2}
\begin{thebibliography}{BDF{\etalchar{+}}22}

\bibitem[ABG04]{ABG04}
S.~Arkhipov, R.~Bezrukavnikov, and V.~Ginzburg, \emph{Quantum groups, the loop
  {G}rassmannian, and the {S}pringer resolution}, J. Amer. Math. Soc.
  \textbf{17} (2004), no.~3, 595--678.

\bibitem[Ach21]{Ach21}
P.~Achar, \emph{Perverse sheaves and applications to representation theory},
  Mathematical Surveys and Monographs, vol. 258, American Mathematical Society,
  Providence, RI, [2021] \copyright 2021.

\bibitem[AG15]{AG15}
D.~Arinkin and D.~Gaitsgory, \emph{Singular support of coherent sheaves and the
  geometric {L}anglands conjecture}, Selecta Math. (N.S.) \textbf{21} (2015),
  no.~1, 1--199.

\bibitem[BBP02]{BBP02}
R.~Bezrukavnikov, A.~Braverman, and L.~Positselskii, \emph{Gluing of abelian
  categories and differential operators on the basic affine space}, J. Inst.
  Math. Jussieu \textbf{1} (2002), no.~4, 543--557.

\bibitem[BDF{\etalchar{+}}22]{BDFRT22}
A.~Braverman, G.~Dhillon, M.~Finkelberg, S.~Raskin, and R.~Travkin,
  \emph{Coulomb branches of noncotangent type (with appendices by {G}urbir
  {D}hillon and {T}heo {J}ohnson-{F}reyd)}, 2201.09475 (2022), 1--41.

\bibitem[Bea99]{BeauvilleSymplecticSingularities}
A.~Beauville, \emph{Symplectic singularities}, Inventiones Mathematicae
  \textbf{139} (1999), 541--549.

\bibitem[Bel23]{BellamyCoulombBranchesHaveSymplecticSingularities}
G.~Bellamy, \emph{Coulomb branches have symplectic singularities}, Lett. Math.
  Phys. \textbf{113} (2023), no.~5, Paper No. 104, 8. \MR{4646952}

\bibitem[BFM05]{BFM05}
R.~Bezrukavnikov, M.~Finkelberg, and I.~Mirkovi{\'c}, \emph{Equivariant
  homology and {$K$}-theory of affine {G}rassmannians and {T}oda lattices},
  Compos. Math. \textbf{141} (2005), no.~3, 746--768.

\bibitem[BFN18]{BFN18}
A.~Braverman, M.~Finkelberg, and H.~Nakajima, \emph{Towards a mathematical
  definition of 3-dimensional $\mathcal{N} = 4$ gauge theories, {II}}, Adv.
  Theor. Math. Phys. \textbf{22} (2018), no.~5, 1071--1147.

\bibitem[BFN19a]{BFN19b}
\bysame, \emph{Coulomb branches of {$3d$} {$\mathcal{N}=4$} quiver gauge
  theories and slices in the affine {G}rassmannian}, Adv. Theor. Math. Phys.
  \textbf{23} (2019), no.~1, 75--166, With two appendices by Braverman,
  Finkelberg, Joel Kamnitzer, Ryosuke Kodera, Nakajima, Ben Webster and Alex
  Weekes.

\bibitem[BFN19b]{BFN19}
\bysame, \emph{Ring objects in the equivariant derived {S}atake category
  arising from {C}oulomb branches}, Adv. Theor. Math. Phys. \textbf{23} (2019),
  no.~2, 253--344, Appendix by Gus Lonergan.

\bibitem[BGG75]{BGG75}
I.~N. Bern\v{s}te\u{\i}n, I.~M. Gelfand, and S.~I. Gelfand, \emph{Differential
  operators on the base affine space and a study of
  {${\mathfrak{g}}$}-modules}, Lie groups and their representations ({P}roc.
  {S}ummer {S}chool, {B}olyai {J}\'{a}nos {M}ath. {S}oc., {B}udapest, 1971),
  Halsted Press, New York-Toronto, Ont., 1975, pp.~21--64.

\bibitem[BKV22]{BKV22}
A.~Bouthier, D.~Kazhdan, and Y.~Varshavsky, \emph{Perverse sheaves on
  infinite-dimensional stacks, and affine {S}pringer theory}, Adv. Math.
  \textbf{408} (2022), Paper No. 108572, 132.

\bibitem[BL94]{BernsteinLuntzEquivariantSheavesandFunctors}
J.~Bernstein and V.~Lunts, \emph{Equivariant sheaves and functors}, Lecture
  Notes in Mathematics, vol. 1578, Springer-Verlag, Berlin, 1994. \MR{1299527}

\bibitem[CW23]{CWig}
S.~Cautis and H.~Williams, \emph{Ind-geometric stacks}, arXiv:2306.03043
  (2023), 1--67.

\bibitem[DG19]{DG19}
T.~Dimofte and N.~Garner, \emph{Coulomb branches of star-shaped quivers}, J.
  High Energy Phys. (2019), no.~2, 004, front matter+87.

\bibitem[DHK21]{DHK21}
A.~Dancer, A.~Hanany, and F.~Kirwan, \emph{Symplectic duality and implosions},
  Adv. Theor. Math. Phys. \textbf{25} (2021), no.~6, 1367--1387.

\bibitem[DKS13]{DKS13}
A.~Dancer, F.~Kirwan, and A.~Swann, \emph{Implosion for hyperk\"{a}hler
  manifolds}, Compos. Math. \textbf{149} (2013), no.~9, 1592--1630.

\bibitem[Gan24]{GannonAProofofGinzburgKazhdanConjecture}
T.~Gannon, \emph{Proof of the {G}inzburg-{K}azhdan conjecture}, Adv. Math.
  \textbf{448} (2024), Paper No. 109701, 28. \MR{4742034}

\bibitem[GJS02]{GJS02}
V.~Guillemin, L.~Jeffrey, and R.~Sjamaar, \emph{Symplectic implosion},
  Transform. Groups \textbf{7} (2002), no.~2, 155--184.

\bibitem[GK22]{GK22}
V.~Ginzburg and D.~Kazhdan, \emph{Differential operators on {$G / U$} and the
  {G}elfand-{G}raev action}, Adv. Math. \textbf{403} (2022), Paper No. 108368,
  48.

\bibitem[GR15]{GR15}
V.~Ginzburg and S.~Riche, \emph{Differential operators on {$G/U$} and the
  affine {G}rassmannian}, J. Inst. Math. Jussieu \textbf{14} (2015), no.~3,
  493--575.

\bibitem[Jia25]{JiaTheGeometryOfTheAffineClosureOfCotangentBundleOfBasicAffineSpaceForSLn}
B.~Jia, \emph{The affine closure of {$T^*({\rm SL}_n/U)$}}, J. Lie Theory
  \textbf{35} (2025), no.~1, 83--100. \MR{4882678}

\bibitem[LS06]{LS06}
T.~Levasseur and J.~T. Stafford, \emph{Differential operators and cohomology
  groups on the basic affine space}, Studies in {L}ie theory, Progr. Math.,
  vol. 243, Birkh\"{a}user Boston, Boston, MA, 2006, pp.~377--403.

\bibitem[Mac23]{Mac23}
M.~Macerato, \emph{Levi-equivariant restriction of spherical perverse sheaves},
  arXiv:2309.07279 (2023), 1--81.

\bibitem[MV07]{MV07}
I.~Mirkovi\'{c} and K.~Vilonen, \emph{Geometric {L}anglands duality and
  representations of algebraic groups over commutative rings}, Ann. of Math.
  (2) \textbf{166} (2007), no.~1, 95--143.

\bibitem[Ras14]{Ras14b}
S.~Raskin, \emph{{D}-modules on infinite-dimensional varieties},
  https://gauss.math.yale.edu/$\sim$sr2532/dmod.pdf (2014), 1--61.

\bibitem[Ric17]{RicheKostantSection}
Simon Riche, \emph{Kostant section, universal centralizer, and a modular
  derived {S}atake equivalence}, Math. Z. \textbf{286} (2017), no.~1-2,
  223--261. \MR{3648498}

\bibitem[Tel22]{Tel22}
C.~Teleman, \emph{Coulomb branches for quaternionic representations},
  arXiv:2209.01088 (2022), 1--31.

\end{thebibliography}

\end{document}